\documentclass[12pt,reqno]{amsart}

\usepackage{fullpage,mymacros}
\usepackage[all]{xy}
\usepackage[
pagebackref,%
bookmarks=true,bookmarksnumbered=true,setpagesize=false,%
 colorlinks=true,%
 pdftitle={},%
 pdfsubject={},%
 pdfauthor={},%
 pdfkeywords={TeX; dvipdfmx; hyperref; color;},%
 colorlinks=false]{hyperref}
\usepackage[
    capitalise,
    noabbrev,
    nameinlink
    ]{cleveref}
\crefname{theorem}{Theorem}{Theorems}
\crefname{corollary}{Corollary}{Corollaries}
\crefname{lemma}{Lemma}{Lemmas}
\crefname{proposition}{Proposition}{Propositions}
\crefname{definition}{Definition}{Definitions}
\crefname{remark}{Remark}{Remarks}
\crefname{example}{Example}{Examples}
\crefname{claim}{Claim}{Claims}
\crefname{sublemma}{Sublemma}{Sublemmas}
\crefname{conjecture}{Conjecture}{Conjectures}
\crefname{assumption}{Assumption}{Assumptions}
\crefname{question}{Question}{Questions}
\crefname{problem}{Problem}{Problems}
\crefname{condition}{Condition}{Conditions}
\crefname{exercise}{Exercise}{Exercises}
\crefname{fact}{Fact}{Facts}
\crefname{notation}{Notation}{Notations}
\crefname{construction}{Construction}{Constructions}
\crefname{recap}{Recap}{Recaps}
\crefname{step}{Step}{Steps}
\crefname{definition-lemma}{Definition-Lemma}{Definition-Lemmas}
\crefname{definition-proposition}{Definition-Proposition}{Definition-Propositions}
\crefname{summary}{Summary}{Summaries}
\crefname{section}{Section}{Sections}
\crefname{table}{Table}{Tables}
\crefname{figure}{Figure}{Figures}
\crefformat{enumi}{(#2#1#3)}
\usepackage{todonotes}
\setuptodonotes{inline}

\newcommand{\an}{\mathrm{an}}
\newcommand{\MM}{\scMbar_{g,n}}

\DeclareMathOperator{\obj}{obj}
\newtheorem*{notationconvention}{Notation and conventions}
\newaliascnt{definition-proposition}{theorem}
\newtheorem{definition-proposition}[definition-proposition]{Definition-Proposition}
\aliascntresetthe{definition-proposition}

\usepackage[
    inline,
    final
    ]{ showlabels }

\title[Noncommutative rigidity]{Noncommutative rigidity of the moduli stack of
stable pointed curves}

\author[S. Okawa]{Shinnosuke  Okawa}
\address{Department of Mathematics,
Graduate School of Science, Osaka University
\newline 1-1, Machikaneyamacho, Toyonaka, Osaka 560-0043, Japan}
\email{okawa@math.sci.osaka-u.ac.jp}

\author[T. Sano]{Taro Sano}
\address{Department of Mathematics,
Graduate School of Science, Kobe university
\newline 1-1, Rokkodai, Nada-ku, Kobe, 657-0001, Japan}
\email{tarosano@math.kobe-u.ac.jp}

\begin{document}
\maketitle

\begin{abstract}
We prove that the second Hochschild cohomology group of the moduli stack of stable $n$-pointed genus $g$ curves
vanishes for all but finitely many \( ( g, n ) \).

\end{abstract}

\tableofcontents

\section{Introduction}
\label{section:introduction}

As an attempt to formulate a non-abelian generalization of the fact that the
\(
   i
\)-th cohomology group vanishes on a space of dimension
\(
   d
\)
for
\(
   i > d
\), Kapranov proposed his geometric syzygy principle in \cite{kapranov} (see also \cite[Section 1.3]{MR3165021} and \cite[Introduction]{MR2460695}): starting with a geometric object
\(
   X _{ 0 }
\)
related to an algebraic variety of dimension
\(
   d
\),
and inductively defining the space
\(
	X _{ i + 1 }
\)
as the moduli space of deformations of
\(
   X _{ i }
\),
the
\(
   d
\)-th step
\(
   X _{ d }
\)
should be rigid. The case
\(
   d = 1
\)
is settled in \cite[Theorem~2.1]{MR2460695}; i.e., the first cohomology group of the tangent bundle of the moduli stack
\(
   \scMbar _{ g, n }
\)
of stable pointed curves vanishes for all \( ( g, n ) \) (in characteristic \( 0 \). See \cite{MR3658203} for the case \( g = 0 \) in positive characteristics).

Based on this, Manin and Smirnov asked in \cite[Section 0.1 (B)]{MR3098789} whether \( \cMbar _{ g, n } \) admits nontrivial \emph{noncommutative} deformations. The question was motivated by the observation that unusual deformations of canonical and rigid objects tend to be interesting, such as \(q\)-deformation and quantum groups. The goal of this paper is to give a somewhat negative answer to this question by showing that \( \cMbar _{ g, n } \) is noncommutatively rigid for all but finitely many \( ( g, n ) \).

By noncommutative deformations we mean the flat deformations of the category
\(
    \Qcoh \scMbar_{g, n}
\)
of quasi-coherent sheaves on
\(
    \scMbar_{g, n}
\)
as an abelian category.
The notion of flat deformation of abelian categories, which naturally encompasses that of schemes, is defined in \cite{MR2344349}.
It is shown in \cite[Theorem~3.1]{MR2183254} that deformations of an abelian category are controlled by its \emph{Hochschild cohomology}, in such a way that the second Hochschild cohomology group serves as the tangent space. For a smooth proper Deligne--Mumford stack
$
X
$
over a field
$
\bfk
$
the
\(
   i
\)-th Hochschild cohomology group of the abelian category
\(
   \Qcoh X
\)
is computed as follows, where
\(
   \Delta
   \colon
   X \to
   X \times _{ \bfk } X
\)
is the diagonal morphism of
\(
   X
\)
relative to
\(
   \Spec \bfk
\):
\begin{equation}
\label{eq:Hochschild_cohomology_for_smooth_proper_DM_stacks}
\HH^i ( X )
=\HH^i_{\bfk} ( X )
\coloneqq
\HH^i_{\bfk} ( \Qcoh X )
\stackrel{\textrm{\cref{theorem:derived_morita_for_dm_stack}}}{\simeq}
\Ext^i_{ X \times_{\bfk} X } ( \Delta _{ * } \cO _{ X } , \Delta _{ * } \cO _{ X } )
\end{equation}
The main results of this paper are summarized as follows.

\begin{theorem}[Main Theorem]\label{theorem:main theorem of the paper}
Suppose that the base field
$
 \bfk
$
is of characteristic \(0\). Then
\begin{enumerate}
\item
\(
	\HH ^{ 2 } ( \Mbar _{ 0, n } )
	=
	0
\)
holds for all
\(
   n
\)
but
\(
   5
\), where
\(
	\dim _{ \bfk } \HH ^{ 2 } ( \Mbar _{ 0, 5 } )
	=
	6
\).

\item
\(
	\HH ^{ 2 } ( \cMbar _{ g, n } )
	=
	0
\)
holds for all
\(
	( g, n )
\)
with
\(
	g \ge 1
\),
except possibly for the following 8 cases: 
\begin{align}\label{eq:remaining_cases}
 ( g, n ) = ( 4, 0 ),
 ( 3, 1 ), ( 3, 0 ),
 ( 2, 2 ), ( 2, 1 ), ( 2 , 0 ),
 ( 1, 3 ), ( 1, 2 )
\end{align}
\end{enumerate}
\end{theorem}

\begin{remark}
\begin{enumerate}[(1)]
    \item 
    \cref{theorem:main theorem of the paper} implies that
\(
   \scMbar _{ g, n }
\)
is noncommutatively rigid for almost all
\( (g, n) \), with at most 9 exceptions. Conceptual reasons for this phenomenon, as well as the relationship to the fact that all noncommutative deformations of curves are actually commutative, are not clear to the authors.

\item
The exceptions~\eqref{eq:remaining_cases} appear for technical reasons, and it is not clear to the authors whether they are essential or not. The remaining obstructions for the noncommutative rigidity in those cases are explicitly described in \cref{corollary:descriptions_of_the_remaining_cases}, as certain cohomology groups on the twisted sectors of the inertia stack of
\(
    \cMbar _{ g, n }
\).

\item
Since
\(
	\Mbar _{ 0, 5 }
\)
is the del Pezzo surface of degree \(5\), one obtains a \(6\)-dimensional global family of noncommutative deformations as blowups of noncommutative projective planes in four points in the sense of \cite{MR1846352} or as certain Artin--Schelter regular algebras whose definition depends on a choice of a helix of the derived category of coherent sheaves of
\(
   \Mbar _{ 0 , 5 }
\) (\cite{2020arXiv200707620O}).
A modular interpretation of these noncommutative deformations in terms of the notion of pointed stable curves, especially the reason why the noncommutative rigidity fails in the case
\(
    ( g, n ) =
   ( 0, 5 )
\), remains open.
\end{enumerate}
\end{remark}

Let us explain the main ideas and methods of the proof of \cref{theorem:main theorem of the paper}. The first ingredient is the Hochschild-Kostant-Rosenberg (HKR) isomorphism, which allows us to compute the Hochschild cohomology groups in terms of sheaf cohomology groups. It was not known for Deligne--Mumford stacks when this paper first appeared on the arXiv, but fortunately it is now established as \cite[Corollary~4.16]{2025arXiv250900501F}, which we recall in \cref{theorem:HKR_for_DM_stack} for the convenience of the reader.
Applying \cref{theorem:HKR_for_DM_stack} for
\(
   X = \cMbar _{ g, n }
\)
and
\(
   i = 2
\)
we get the following decomposition:
\begin{equation}\label{eq:Intro_HKR_for_DM_stack}
	\HH ^2 ( \cMbar _{ g, n } )
	\simeq
	\bigoplus _{ W \subset I \cMbar _{ g, n } }
	\bigoplus _{ p + q = 2 }
	H ^{ q - c _W } \lb W, \wedge ^{ p } \Theta _{ W / \bfk } \otimes \wedge ^{ c_W } \cN _{ W / \cMbar _{ g, n } }\rb
\end{equation}
In \eqref{eq:Intro_HKR_for_DM_stack}, \( W \) runs through the connected components of the inertia stack \( I \cMbar _{ g, n } \). Note that \( I \cMbar _{ g, n } \) is a union of the main component \( \cMbar _{ g, n } \subset I \cMbar _{ g, n } \) and  the twisted sectors (cf. \cref{definition-lemma:inertia_stack_basic_facts}).
The symbols \( c _{ W } \) and \( \cN _{ W / \cMbar _{ g, n } } \) denote the codimension and the normal bundle of the natural morphism \( W \to \cMbar _{ g, n } \), respectively.
The sheaf
$
\Theta_{W/\bfk}
=
 \lb \Omega ^{ 1 } _{ W / \bfk } \rb ^{ \vee }
$ 
is the locally free sheaf of algebraic vector fields on
$
W
$
over
$\bfk$.

In order to show the vanishing of \( \HH ^{ 2 } ( \cMbar _{ g, n } ) \) using~\eqref{eq:Intro_HKR_for_DM_stack}, in~\cref{section:The_classification_of_twisted_sectors_of_codimension_at_most_2} we classify the connected components \( W \subset I \cMbar _{ g, n } \) with \( c _{ W } = 0, 1, 2 \) for all \( ( g, n ) \).

In~\cite{MR2460695} Hacking used a vanishing theorem to show the vanishing of
\(
   H ^{ 1 }
   \left(
    \scMbar _{ g, n },
    \Theta
    _{
        \scMbar _{ g, n }        
    }
   \right)
\),
together with the inductive structure of the boundary strata of
\(
   \scMbar _{ g, n }
\)
and the positivity of standard line bundles. Our proof of the vanishing of cohomology groups in the right hand side of \eqref{eq:Intro_HKR_for_DM_stack} is inspired by Hacking's method, but it differs in several points.
The most notable one is the difference in vanishing theorems used in the proof. In \cite{MR2460695}, a Kodaira vanishing theorem on Deligne--Mumford stacks \cite[Theorem A.1]{MR2460695} is established and used. Since this is not enough for our purpose, we establish the logarithmic Kodaira--Akizuki--Nakano vanishing theorem for normal crossing pairs of Deligne--Mumford stacks:

\begin{theorem}[{\(=\)\cref{theorem:logarithmic_KAN} in characteristic \(0\)}]\label{theorem:Introduction logarithmic_KAN}
	Let
	$
	\cX
	$
	be a smooth proper Deligne--Mumford stack of dimension
	$d$
	over a field $\bfk$ of characteristic \(0\),
	$
	\cD
	$
	a normal crossing divisor on
	$
	\cX
	$,
	and
	$\cL$ an ample line bundle on
	$
	\cX
	$. 
	Then we have
	\begin{equation}
	 H^i(\cX, \Omega^j_\cX (\log \cD)(-\cD) \otimes \cL) =0 
	\end{equation}
	if $i +j > d$.
\end{theorem}
The proof of \cref{theorem:logarithmic_KAN} is done by generalizing the mod \( p \) argument established in \cite{MR894379} for schemes, after \cite{MR3060750}. We believe that \cref{theorem:logarithmic_KAN} should be of independent interest and useful further applications.

After these preparations we show the following theorem, which asserts that there is no contribution from the main component \( \cMbar _{ g, n } \subset I \cMbar _{ g, n } \) in the right hand side of \eqref{eq:Intro_HKR_for_DM_stack} except for the case \( ( g, n ) = ( 0, 5 ) \).
\begin{theorem}\label{theorem:vanishing_of_the_untwisted_sectors}
Suppose that
$
 \bfk
$
is of characteristic \(0\). Then
\begin{enumerate}
\item
$
 H ^{ 2 } \lb \scMbar_{g, n}, \cO _{ \scMbar_{g, n} } \rb = 0
$
holds for any
$
 ( g, n )
$.

\item
\(
	\dim _{ \bfk }
	H ^{ 0 } \lb \scMbar_{g, n}, \wedge ^{ 2 } \Theta _{ \scMbar_{g, n} } \rb
	=
\begin{cases}
	0 & ( g, n ) \neq ( 0, 5 )\\
	6 & ( g, n ) = ( 0, 5 ).
\end{cases}
\)
\end{enumerate}
\end{theorem}
Since
$
 \scMbar _{ g, n }
$
is a scheme if and only if $g = 0$, in which case there are no twisted sectors,
the Hacking's rigidity theorem and \cref{theorem:vanishing_of_the_untwisted_sectors} immediately imply the first assertion of \cref{theorem:main theorem of the paper}.

The first assertion of
\cref{theorem:vanishing_of_the_untwisted_sectors} is shown in \cref{theorem:vanishing_of_O}.
In fact it also follows from
\cite[Theorem 2.2]{MR1733327} plus the Hodge decomposition
for smooth proper Deligne--Mumford stacks (\cref{lemma:log_hodge_decomposition}).
Our proof can be seen as a shortcut version. 

The second assertion of
\cref{theorem:vanishing_of_the_untwisted_sectors} is shown in \cref{section:proof_of_no_bivector_on_Mgnbar} as \cref{theorem:no_bivector_on_Mgnbar} for the case
\(
   ( g, n )
   \neq
   ( 0, 5 ),
   ( 1, 2 )
\). The cases
\(
   ( g, n )
   =
   ( 0, 5 ),
    ( 1, 2 )
\)
are settled in~\cref{section:exceptional_cases}
as \cref{lemma:bivectors_on_M05} and \cref{proposition:bivectors_on_M12}, respectively.
The proof of
\cref{theorem:no_bivector_on_Mgnbar}
is based on the inductive structure of the moduli stack
$
\scMbar_{g,n}
$
(\cref{lemma:inductive_structure_of_the_boundary}),
ampleness of the log canonical bundle
$
\omega_{\scMbar_{g,n}}(\cB)
$
(\cref{proposition:ampleness_of_the_log_canonical_divisor}),
and the positivity of the $\psi$-classes
(\cref{lemma:psi_classes_are_nef_and_big}),
to which we apply \cref{theorem:Introduction logarithmic_KAN}.

To settle the cases
\(
   g \ge 1
\), we need to show the vanishing of the contributions from the twisted sectors in the right hand side of \eqref{eq:Intro_HKR_for_DM_stack}.
In this paper we prove it except possibly for the cases
$
 g + n \le 4
$
and
$
 ( g, n ) \neq ( 1, 1 )
$;
these are the 8 cases in \eqref{eq:remaining_cases}:

\begin{theorem}\label{theorem:introduction_vanishing}
Suppose that
$
 \bfk
$
is of characteristic \(0\). Then the right hand side of \eqref{eq:Intro_HKR_for_DM_stack}
vanishes for any \( ( g, n ) \) with \( g \ge 1\), except possibly for the 8 cases in
\eqref{eq:remaining_cases}.
\end{theorem}

The proof of \cref{theorem:introduction_vanishing} is quite involved, and divided into two steps. The first step is in \cref{section:The_classification_of_twisted_sectors_of_codimension_at_most_2}, where we give a complete classification of the connected components
$
 W \subset I \scMbar _{ g, n }
$
with \( c_{ W } = 0, 1,\) or \(2\)
for all
$
 ( g, n )
$.
They are given in
\cref{proposition:classification_of_twisted_sectors_of_codim_0},
\cref{summary:classification_of_twisted_sectors_of_codim_1},
and
\cref{summary:classification_of_twisted_sectors_of_codim_2}, respectively, inductively on \( c _{ W }\).
These are the components which appear in the right hand side of \eqref{eq:Intro_HKR_for_DM_stack}.
The general pattern is that there are twisted sectors of common type which appear for almost all
$
 ( g, n )
$, and then there are exotic ones which appear only for small
$
 ( g, n )
$.

The second step is in \cref{section:The_vanishing_of_contributions_from_the_twisted_sectors}, where we prove the vanishing of the contributions from the connected components of common type by using the modular interpretation of the stratification of \( \cMbar _{ g, n } \). This part is similar to the proof of \cref{theorem:no_bivector_on_Mgnbar}, but more involved.
When
$
 g + n \ge 5
$,
no exotic component shows up, so we can conclude the vanishing of the right hand side of \eqref{eq:Intro_HKR_for_DM_stack}.
By dimension reasons we can also settle
$
 \scMbar _{ 1, 1 }
$,
being left with the 8 cases \eqref{eq:remaining_cases}.

For the cases \eqref{eq:remaining_cases}, in \cref{corollary:descriptions_of_the_remaining_cases} we explicitly describe the contributions from the exotic components which remain in the right hand side of \eqref{eq:Intro_HKR_for_DM_stack}.
In order to prove the vanishing of these contributions, we need to investigate the negativity of the normal bundles of those exotic loci of
$
 \scMbar _{ g, n }
$, such as the divisor of hyperelliptic curves in
$
 \scMbar _{ 3, 0 }
$
and so on.
We leave them as interesting open questions on the geometry of
$
 \scMbar _{ g, n }
$.

\subsection*{Acknowledgements}
The authors would like to thank Yuri Manin for asking them
about noncommutative deformations of
$
\Mbar_{0,n}
$ and for providing them with the unpublished note
\cite{kapranov}.
They would also like to thank Osamu Fujino for the useful information
about Hodge theory on not-necessarily-simple normal crossing pairs, and Kazushi Ueda for quite useful discussions about $\scMbar_{1,2}$.
They are also indebted to him and Luca Tasin for pointing out that the HKR isomorphism is more subtle in the case of Deligne--Mumford stacks, which was overlooked in the first preprint version of this paper.
They appreciate the Max-Planck-Institute f\"ur Mathematik
for their hospitality and support. This work was initiated during the authors' stay there.
During the preparation of this paper, S.O. was partially supported by JSPS Grant-in-Aid for Young Scientists No.~25800017,
Grants-in-Aid for Scientific Research
(16H05994,
16K13746,
16H02141,
16K13743,
16K13755,
16H06337,
18H01120)
and the Inamori Foundation.
T.S. was partially supported by JSPS Grant-in-Aid for Young Scientists No.~JP16K17573 and Grant-in-Aid for Scientific Research JP17H06127.

\begin{notationconvention}
Throughout the paper we work over a field $\bfk$, whose characteristic is assumed to be 0 unless otherwise stated.
An \emph{elliptic curve} is a smooth projective curve of genus $1$ with a marked point,
which will be tacitly assumed to be the origin of the group structure. A \emph{negation} of an elliptic curve is the automorphism induced by the negation by the group structure.
\emph{HE} stands for hyperelliptic. A \emph{Weierstrass point} is a fixed point of a HE involution.
The \emph{Weierstrass divisor} of
$
 \scMbar _{ 2, 1 }
$
is the divisor whose general point represents a smooth curve of genus 2 whose marked point is a Weierstrass point.
For a Deligne--Mumford stack $X$,
$
 \nu \colon X ^{ \nu } \to X
$
denotes its normalization.
For a morphism $f \colon X \to Y$ of Deligne--Mumford stacks,
$
 I f \colon I X \to I Y
$
denotes the induced morphism between the inertia stacks.
\end{notationconvention}

\section{Preliminaries}

\subsection{Recap of $ \scMbar_{g,n}$}
We recall some standard facts
about the moduli of stable pointed curves.
For statements without proofs, see \cite{MR2460695} and references therein.

Given a pair of integers
$ ( g , n ) $
satisfying the three conditions
\begin{equation}\label{eq:condition_for_Mbar_g_n_nonempty}
n \ge 0 ,
\quad
g \ge 0 ,
\quad
2g - 2 + n > 0
\left(
    \iff
   ( g, n )
   \neq
   ( 0, 0 ),
   ( 0 , 1 ),
   ( 0, 2 ),
   ( 1, 0 )
\right)
\end{equation}
we have the category
$
\scMbar_{ g , n }
$
of stable
$
n
$-pointed genus
$
g
$
curves fibered in groupoids over the category
$
( \Sch / \bfk )
$.
It is a smooth proper Deligne--Mumford stack of dimension
$
3g-3 + n.
$
There exists the universal curve
$
 \pi \colon \cU_{g , n} \to \scMbar_{ g , n },
$
which itself is a smooth proper Deligne--Mumford stack,
and for
$
i = 1 , 2 , \dots , n
$
we denote by
$
 \sigma_i \colon \scMbar_{ g , n } \to \cU_{g , n}
$
the section of $ \pi $ representing the $ i $-th marked point.
With these notations, the
\emph{$\psi$-line bundles} on
$
\scMbar_{g,n}
$
are defined as
$
\psi_{ i }
\coloneqq
(\sigma_{i})^* \omega_{\pi}
$
for
$
i=1,2,\dots, n
$.

\begin{lemma}\label{lemma:psi_classes_are_nef_and_big}
Suppose
$
n>0
$.
Then for any
$
i=1,2,\dots, n
$,
the line bundles
$
\psi_i
$
are nef and big.
\end{lemma}

\begin{proof}
For $i=n$, the assertion follows from
\cite[Theorem 3.2 and Lemma 4.3]{MR2460695}
(see also
\cite[Section 4]{MR1680559})).
For other $ i $,
use the fact that the symmetry group
$
\frakS_n
$
acts on
$
\scMbar_{g,n}
$
in such a way that the transposition
$
( i , n )
$
sends
$
\psi_i
$
to
$
\psi_n
$.
\end{proof}

There exists a closed substack
$
 \cB \subset \scMbar_{g,n}
$
representing singular stable pointed curves.
It is a reduced normal crossing divisor.
The open substack of smooth pointed stable curves will be
denoted by
$
 \cM_{g,n} \coloneqq \scMbar_{g,n} \setminus \cB.
$

The coarse moduli space 
$ p \colon \scMbar_{g,n} \to \Mbar_{g,n}
$
is a morphism to a projective variety $\Mbar_{g,n}$. 
When
$
g = 0
$,
the morphism $ p $
is an isomorphism.
The following positivity is an essential ingredient of our proof of
\cref{theorem:no_bivector_on_Mgnbar}.

\begin{proposition}
\label{proposition:ampleness_of_the_log_canonical_divisor}
For any
$
( g , n )
$
satisfying
\eqref{eq:condition_for_Mbar_g_n_nonempty},
the log canonical bundle
$
 \omega_{\scMbar_{g,n}} ( \cB )
$
is ample.
\end{proposition}

\begin{proof}
This is a special case of
\cite[Theorem 4.1]{MR2831833} where
\(
	a _{ 1 } = a _{ 2 } = \cdots = a _{ n } = 1
\).
\end{proof}

We also use the inductive nature of the boundary divisors.

\begin{lemma}\label{lemma:inductive_structure_of_the_boundary}
Let
$
\cB'
$
be a connected component of the normalization
$
\cB^{1} \twoheadrightarrow \cB
$, and $\cN$ be the normal bundle of the morphism
$
 \cB'  \to \scMbar_{g,n}
$.
Then
$
 \cB '
$
 admits a finite surjective \'etale morphism
$
\rho \colon
\cC'
\twoheadrightarrow
\cB'
$
from a smooth Deligne--Mumford stack $\cC'$ which satisfies one of the following properties (hence $\cB'$ is also smooth).

\begin{enumerate}[(i)]
\item\label{it:inductive_structure_of_the_boundary;reducible}
$
\cC'
$
is isomorphic to
$
\scMbar_{g' , n' } \times
\scMbar_{g'' , n'' }
$, where
$
g' + g'' = g
$
and
$
n' + n'' = n + 2
$, and
$
 \rho
$
is of degree $2$ when $g' = g''$ and $n'=n''=0$, and an isomorphism otherwise. Moreover, the pullback of $\cN^{\vee}$ to $\cC'$ is sent to
$
 \pr_1^* \psi_{n'} \otimes \pr_2^* \psi_{n ''}
$
under the isomorphism.

\item\label{it:inductive_structure_of_the_boundary;irreducible}
$
\cC'
$
is isomorphic to
$
\scMbar_{g-1 , n+2}
$
and
$
 \rho
$
is of degree two. Moreover, the pullback of $\cN^{\vee}$ to $\cC'$ is sent to
$
 \psi_{n+1} \otimes \psi_{n+2}
$
under the isomorphism.
\end{enumerate}
\end{lemma}

\begin{proof}
See
\cite[p. 813, 814]{MR2460695}.
\end{proof}

As implicitly stated in \cref{lemma:inductive_structure_of_the_boundary},
irreducible components of the boundary divisor $\cB$ are classified as follows.

\begin{definition}\label{definition:boundarydivisors}
Let $g, n \ge 0$ be as in \eqref{eq:condition_for_Mbar_g_n_nonempty}. Let
$
 [n] \coloneqq \{1, \ldots, n \}
$.
For $0 \le j \le g$ and $S \subset [n]$, 
the irreducible divisor
$
 \cB _{ j : S } \subset \scMbar_{g,n}
$
is defined to be the closure of the locus of nodal curves $C_1 \cup C_2$ with one node $p$, where 
$C_1$ is a smooth curve of genus $j$ with marked points $\{p_i \mid i \in S, \ p_i \neq p \}$ and $C_2$ is a smooth curve of genus 
$g-j$ with marked points $\{ p_i \mid i \in [n] \setminus S, p_i \neq p \}$.
Also we let $\cB_0 \subset \scMbar_{g,n}$ denote the closure of the locus of irreducible nodal curves $C$ with one node.
Both $\cB _{ j : S }$ and $\cB _{ 0 }$ are known to be normal crossing divisors, from the deformation theory of stable curves (\cite[p.~81]{MR0262240}).
\end{definition} 

\subsection{Inertia stack}\label{section:inertia_stack}

We also need to understand the inertia stack of
$
 \scMbar _{ g, n }
$.

\begin{definition}\label{definition:inertia_stack}
The \emph{inertia stack}
$
 IX
$
of a Deligne--Mumford stack
$
 X
$
separated over a field
$
 \bfk
$
is defined as the fiber product described in the diagram below.

\begin{align*}
\xymatrix{
\ar @{} [dr] |{\square}
IX \ar[r]^{\pr_2} \ar[d]_{\pr_1}
&
X \ar[d]^{\Delta}
\\
X \ar[r]_{\Delta}
&
X \times X}
\end{align*}
\end{definition}

We recall some basic properties of the inertia stack.

\begin{definition-lemma}
\label{definition-lemma:inertia_stack_basic_facts}
\begin{enumerate}
 \item\label{it:another_description_of_inertia_stack}
 As a category fibered in groupoids over
 $
  ( \Sch / \bfk )
 $,
 $IX$ is equivalent to the category defined as follows; by an abuse of notation, we use the same symbol $IX$ for the category.
 \begin{itemize}
  \item
 An object is a pair
 $
  ( x, g )
 $,
 where
 $
  x \in \obj ( X )
 $
 and
 $
  g \in \Aut { ( x ) }
 $.
 \item
  A morphism from
  $
   ( x, g )
  $
  to
  $
   ( y, h )
  $
  is a morphism
  $
   f \colon x \to y
  $
  which makes the following diagram commutative.
 
\begin{align*}
\xymatrix{
\ar @{} [dr] |{\circlearrowleft}
x \ar[r]^{f} \ar[d]_{g}
&
y \ar[d]^{h}
\\
x \ar[r]_{f}
&
y
}
\end{align*}
\item
The structure functor
$
 IX \to ( \Sch / \bfk )
$
is the composition of the forgetful functor
$
 F \colon ( x, g ) \mapsto x
$
and the structure functor
$
 X \to ( \Sch / \bfk )
$.
\end{itemize}

The stack defined above is the fiber product via the morphisms
$
 \pr _1 = \pr _2 = F
$.

 \item
 The canonical morphism
 $
  \id_X \times \id_X \colon X \to IX
 $
 is a section for both
 $
  \pr_1
 $
 and
 $
  \pr_2
 $.
 Moreover is an isomorphism onto its image, which is a connected component of
 $
  X
 $
 and will be called the \emph{untwisted sector} of
 $
  IX
 $.
 \item
 A connected component
 $
  W \subset IX \setminus X
 $
 is called a \emph{twisted sector}.

\item
 If
 $
  X
 $
 is smooth, so is $IX$.

 \item
 For each twisted sector
 $
  W
 $,
 the canonical morphism
 $
  W \to IX \xto[]{F} X
 $
 factors into an \'etale finite morphism onto its image followed by a closed immersion of the image into
 $
  X
 $.
 
\end{enumerate}
\end{definition-lemma}
\begin{proof}
See \cite{MR1005008} and \cite{Abramovich-Graber-Vistoli}, and \cite[Section 4.1]{pagani2009chen}.
\end{proof}

\begin{lemma}\label{lemma:inertia_commutes_with_product}
Let
$
 X, Y
$
be Deligne--Mumford stacks which are separated over $\bfk$. Then there is a canonical isomorphism of stacks
\begin{align}\label{eq:inertia_commute_with_product}
 I \lb X \times Y \rb \simeq I X \times I Y.
\end{align}
\end{lemma}

\begin{proof}
Note first that the product
$
 X \times Y
$
is explicitly given as the category fibered in groupoids over $\Sch / \bfk$ which is described as follows.
\begin{itemize}
\item
An object is a pair
$
 ( x, y )
$
of objects
$
 x \in X
$
and
$
 y \in Y
$
over the same base scheme.

\item
An morphism from
$
 ( x, y )
$
to
$
 ( x ', y ' )
$
is a pair
$
 \lb f, g \rb
$,
where
$
 f \in X ( x, x ' )
$
and
$
 g \in Y ( y, y ' )
$.
\end{itemize}

By using the explicit model for inertia stacks as described in \cref{definition-lemma:inertia_stack_basic_facts} \cref{it:another_description_of_inertia_stack}, an isomorphism of the categories between
$
 I \lb X \times Y \rb
$
and
$
 I X \times I Y
$
is explicitly given as follows: an object
$
 \lb ( x, y ), ( g, h ) \rb \in I \lb X \times Y \rb
$
corresponds to the object
$
 \lb ( x, g ), ( y, h ) \rb \in I X \times I Y
$
and vice versa. Similarly, a morphism
$
 \lb \varphi, \psi \rb \colon
 \lb ( x, y ), ( g, h ) \rb
 \to
 \lb ( x ', y ' ), ( g ', h ' ) \rb
$
corresponds to
$
 \lb \lb \varphi \colon ( x, g ) \to ( x ' , g ' ) \rb, \psi \colon ( y, h ) \to ( y ', h ' ) \rb
$
and vice versa.
\end{proof}

\begin{lemma}\label{lemma:closed_immersion_is_fully_faithful}
Let
$
 i \colon X \to Y
$
be a closed immersion of Deligne--Mumford stacks separated over $\bfk$.
Then $i$, as a functor, is fully faithful.
\end{lemma}

\begin{proof}
	This follows from
	\cite[\href{https://stacks.math.columbia.edu/tag/0504}{Tag 0504}]{stacks-project}
	and
	\cite[\href{https://stacks.math.columbia.edu/tag/04ZZ}{Tag 04ZZ}]{stacks-project}.
\end{proof}

\begin{lemma}\label{lemma:closed_immersion_is_a_one_monomorphism}
Let
$
 i \colon X \to Y
$
be a closed immersion of Deligne--Mumford stacks separated over $\bfk$.
Then for all morphisms of Deligne--Mumford stacks
$
 f \colon Z \to X
$
and
$
 g \colon W \to X
$,
the natural morphism
$
 F \colon Z \times _{ f, X, g } W \to Z \times _{ i \circ f, Y, i \circ g } W
$
is an isomorphism of stacks.
\end{lemma}

\begin{proof}
By \cref{lemma:closed_immersion_is_fully_faithful}, the functor $i$ is fully faithful.
By using this, we construct a quasi-inverse
$
 G \colon Z \times _{ i \circ f, Y, i \circ g } W \to Z \times _{ f, X, g } W
$
to $F$.

An object of
$
 Z \times _{ i \circ f, Y, i \circ g } W
$
is a triple
$
 \lb z, w, \theta \colon i ( f ( z ) ) \simto i ( g ( w ) ) \rb
$.
Since $i$ is fully faithful, there is a unique
$
 \theta ' \colon f ( z ) \to g ( w )
$
such that
$
 i ( \theta ' ) = \theta
$.
We let
$
 G \lb \lb z, w, \theta \rb \rb \coloneqq \lb z, w, \theta ' \rb
$.
One can also easily define the action of $G$ on morphisms, and can easily verify that
$
 G
$
is a quasi-inverse of $F$. Thus we conclude the proof.
\end{proof}

\begin{proposition}
Let
$
 i \colon X \to Y
$
be a closed immersion of Deligne--Mumford stacks separated over $\bfk$.
Suppose that a connected component
$
 W \subset IY
$
satisfies the property that the natural morphism
$
 F \colon W \to Y
$
factors through $i$. Then there is a unique (up to natural isomorphism) open and closed immersion
$
 \Fbar \colon W \to IX
$
such that
$
 \lb W \hookrightarrow I Y \rb = \lb I X \to I Y \rb \circ \Fbar
$.
\end{proposition}

\begin{proof}
Consider the following diagram, where
$
 Z = Y \times Y
$.
\begin{align}
 \xymatrix{
 X \times _{ Z } X \ar[r] \ar[d] & Y \times _{ Z } Y \ar[r] \ar[d] & Z \ar[d] ^{ \Delta _{ Z } }\\
 X \times X \ar [r] _{ i \times i } & Y \times Y \ar[r] _{ \Delta _{ Y } \times \Delta _{ Y } } & Z \times Z
 }
\end{align}
Since both the outer square and the right square are cartesian, so is the left one.
Since
$
 X \times X \to Y \times Y
$
is a closed immersion, it then follows that so is the left upper horizontal morphism.

On the other hand, since
$
 X \times X \to Y \times Y
$
is a closed immersion of Deligne--Mumford stacks separated over $\bfk$, by \cref{lemma:closed_immersion_is_a_one_monomorphism} there is a natural isomorphism
$
 I X = X \times _{ X \times X } X
 \simto
 X \times _{ Z } X
$,
under which the natural morphism
$
 I X
 \to
 I Y
$
is identified with the left upper horizontal morphism. Hence it is a closed immersion.

Let
$
 F ' \colon W \to X
$
be the morphism such that
$
 F = i \circ F '
$.
This yields the morphism
$
 F ' \times _{ Y \times Y } F ' \colon W \to X \times _{ Y \times Y } X
$.
This uniquely lifts to a morphism
$
 \Fbar \colon W \to X \times _{ X \times X } X = I X
$
with the desired properties. To see that $\Fbar$ is open and closed, note that
$
 \lb W \hookrightarrow I Y \rb = \lb I X \to I Y \rb \circ \Fbar
$
has the same properties and that
$
 I X \to I Y
$
is representable and separated, as we saw above.
\end{proof}

\subsection{Hochschild cohomology of Deligne--Mumford stacks}

The notion of \emph{flat deformation of abelian categories}
is introduced in \cite{MR2238922}.
In \cite{MR2183254}, it is proven that flat deformations of
a
$
 \bfk
$-linear Grothendieck category
$
 \cA
$
is controlled by the Hochschild dgla
$
 \bsC _{ \ab } ( \cA ) [ 1 ]
$.
In particular, the cohomology space,
$
 \cH ^{ 1 } \lb \bsC _{ \ab } ( \cA ) [ 1 ] \rb \eqqcolon H\bsC _{ \ab } ^2 ( \cA )
$,
is in bijection with the set of first order deformations of
$
 \cA
$
up to equivalence.

If
$
 \cA
$
is locally Noetherian, i.e., generated by Noetherian objects, then
there is essentially no difference between deformations of
$
 \cA
$
and those of the full subcategory of Noetherian objects.
(see \cite[Theorem 8.8]{MR2238922}).
This in particular applies to the categories
$
 \coh X \subset \Qcoh X
$,
where
$
 X
$
is a separated and Noetherian Deligne--Mumford stack over a field
$
 \bfk
$.

There is a quasi-isomorphism of Hochschild complexes
$
 \bsC \lb \bfD ( \Qcoh X ) \rb \simto \bsC _{ \ab } ( \Qcoh X )
$,
where
\begin{align}
 \bsC \lb \bfD ( \Qcoh X ) \rb
 \coloneqq
 \bR \underline{\Hom} _{ \bfD ( \Qcoh X ) \otimes ^{ \bL } \bfD ( \Qcoh X ) ^{ \op } } \lb \bfD ( \Qcoh X ), \bfD ( \Qcoh X ) \rb
\end{align}
is the Hochschild complex of (a dg-enhancement of) the  unbounded derived category of quasi-coherent sheaves on $X$ as defined in \cite[Section 8.1]{Toen_HTDGC} (see \cite[p. 1558 (3)]{MR2474321}). It is shown in \cite[Corollary 8.1]{Toen_HTDGC} that
\begin{align}
 \bsC \lb \bfD ( \Qcoh X ) \rb
 \simeq
 \bR \underline{\Hom} \lb \bfD ( \Qcoh X ), \bfD ( \Qcoh X ) \rb \lb \id, \id \rb.
\end{align}
On the other hand, consider the complex
\begin{align}
 \bsC _{ \Swan } ( X )
 \coloneqq
 \Hom ^{ \bullet } _{ X \times X } \lb \Delta _{ * } \cO _{ X }, \Delta _{ * } \cO _{ X } \rb.
\end{align}

\begin{theorem}\label{theorem:derived_morita_for_dm_stack}
Let
$
 X
$
be a separated and quasi-compact Deligne--Mumford stack over a field
$
 \bfk
$ of characteristic \(0\).
Then there is an isomorphism
$
 \bsC \lb \bfD ( \Qcoh X ) \rb \simeq \bsC _{ \Swan } ( X )
$
in the homotopy category of complexes of
$
 \bfk
$-vector spaces.
\end{theorem}

The following proof of \cref{theorem:derived_morita_for_dm_stack} is obtained independently of the similar but more thorough \cite[Proposition A.1]{MR4057490}.

\begin{proof}
In the case when
$
 X
$
is a scheme, the assertion is shown in \cite[Theorem 8.9]{Toen_HTDGC}.
In fact, the isomorphism
\begin{equation}\label{eq:Toen}
 \bR \underline{\Hom} _{ c } ( L _{ \Qcoh } ( X ), L _{ \Qcoh } ( X ) ) \simto
 L _{ \Qcoh } ( X \times X )
\end{equation}
of
\cite[Theorem 8.9]{Toen_HTDGC}
in the homotopy category of dg-categories induces an isomorphism of dg
$
 \bfk
$-modules
\begin{equation}
 \bR \underline{\Hom} ( L _{ \Qcoh } ( X ), L _{ \Qcoh } ( X ) ) ( \id, \id )
 \simto
 L _{ \Qcoh } ( X \times X ) ( \Delta _{ * } \cO_X, \Delta _{ * } \cO_X )
\end{equation}
in the homotopy category of complexes of $\bfk$-vector spaces. The LHS of \eqref{eq:Toen} denotes the dg category of dg endofunctors commuting with arbitrary coproducts, which obviously contains
$
 \id
$,
and
$
 L _{ \Qcoh } X
$
is the notation for
$
 \bfD ( \Qcoh X )
$
in \cite{Toen_HTDGC}.

The proof of \cite[Theorem 8.9]{Toen_HTDGC} can be extended to the case of Deligne--Mumford stacks modulo some modifications. In the arguments, some properties of the compact generators of
$
 \bfD ( \Qcoh X )
$
of a quasi-compact and separated scheme
$
 X
$
are used. It is enough to check that these properties are fulfilled by Deligne--Mumford stacks.

One of them is that
$
 \bfD ( \Qcoh X )
$
admits a compact generator. By \cite[Theorem 1.2]{2014arXiv1405.1888H} and \cite[Theorem A]{MR3705292} this generalizes to quasi-compact and separated Deligne--Mumford stacks $X$ over $\bfk$.
The other property is about the exterior product of compact generators on the product stacks, which we treat as the following \cref{lemma:bondal_vandenberg_for_stack}.
\end{proof}

\begin{lemma}[{$=$\cite[Lemma 3.4.1]{Bondal-van_den_Bergh}
for Deligne--Mumford stacks}]\label{lemma:bondal_vandenberg_for_stack}
Let
$
 X, Y
$
be quasi-compact and separated Deligne--Mumford stacks over $\bfk$.
Assume that
$
 E, F
$
are compact generators of
$
 \bfD ( \Qcoh X )
$
and
$
 \bfD ( \Qcoh Y )
$. Then
$
 E \boxtimes F \in \bfD ( \Qcoh X \times Y )
$
is again a compact generator.
\end{lemma}
\begin{proof}
Take atlases
$
 U \to X
$
and
$
 V \to Y
$
with
$
 U, V
$
affine and consider the flat morphism
$
 \pi \colon U \times V \to X \times Y
$.
Take any
$
 G \in \bfD ( \Qcoh X \times Y )
$
with
$
 \Hom ( E \boxtimes F , G [ m ] ) = 0
$
for all
$
 m \in \bZ
$.
Then by the same arguments as in the proof of
\cite[Lemma 3.4.1]{Bondal-van_den_Bergh}, using that
$
 U, V
$
are affine, one can show that
$
 \pi ^{ * } G = 0
$.
Since
$
 \pi
$
is flat, for each
$
 i \in \bZ
$
one has
$
 \pi ^{ * } \cH ^{ i } ( G )
 \simeq
 \cH ^{ i } \lb \pi ^{ * } G \rb
 =
 0
$,
so that
$
 \cH ^{ i } ( G ) = 0
$
by the surjectivity of $\pi$. Thus we conclude the proof.
\end{proof}
We next recall the
\emph{Hochschild-Kostant-Rosenberg isomorphism}
(HKR isomorphism for short) for Swan's Hochschild cohomology
$
 \bsC _{ \Swan }
$, as established in~\cite{2025arXiv250900501F}:

\begin{theorem}[{\(=\)\cite[Corollary~4.16]{2025arXiv250900501F}}]\label{theorem:HKR_for_DM_stack}
Let
$
 X
$
be a smooth proper Deligne--Mumford stack over a field
$
 \bfk
$
of characteristic
$
 0
$.
Then there is an isomorphism
\begin{equation}\label{eq:HKR_for_DM_stack}
 H\bsC _{ \Swan }^i ( \Qcoh X )
 \simto
 \bigoplus _{ W \subset IX }
 \bigoplus _{ p + q = i }
 H ^{ q - c _W } \lb W, \wedge ^{ p } \Theta _{ W / \bfk } \otimes \wedge ^{ c_W } \cN _{ W / X }\rb,
\end{equation}
where
\begin{itemize}
 \item
 $
  W \subset IX
 $
 is either the untwisted or a twisted sector of the inertia stack
 $
  IX
 $
 of
 \(
    X
 \),
 \item
 $
  c _{ W } = \dim X - \dim W \ge 0
 $, and
 \item
 $
  \cN _{ W / X }
 $
 is the normal bundle of the morphism
 $
  W \to X
 $.
\end{itemize}
\end{theorem}
By
 \cref{definition-lemma:inertia_stack_basic_facts},
$
 \cN _{ W / X }
$
is a locally free sheaf of rank
$
 c_W
$.
When
$
 X
$
is a scheme, then there are no twisted sectors and we obtain the original HKR isomorphism.
\begin{remark}\label{remark:it is now a theorem!}
    In a previous version of this paper, we stated \cref{theorem:HKR_for_DM_stack} as a conjecture and the main results for the cases
    \(
       g > 0
    \)
    were conditional on it. 
\end{remark}

\section{Logarithmic differential forms and the Kodaira--Akizuki--Nakano vanishing theorem on Deligne--Mumford stacks}
\label{section:log_differential_forms} 

For a variety $X$ over $\bfk$,
the $p$-th exterior product of the sheaf of K\"ahler differentials will be denoted by
$\Omega^p_X$. Its double dual $(\Omega^p_X)^{**}$
will be denoted by $\tilde{\Omega}^p_X$.

\begin{lemma}\label{lemma:Kahler_differential_under_quotient}
	Let $Y$ be a smooth affine variety over a field
    \(
       \bfk
    \)
    of characteristic \(0\), and $G$ a finite group acting on $Y$. 
	Let $X\coloneqq Y/G$ be the quotient by the $G$-action and $\pi \colon Y \to X$ the quotient morphism. 
Then we have a canonical isomorphism 
 \[
 \tilde{\Omega}^p_X \simto (\pi_* \Omega^p_Y)^G,  
 \]
 where the RHS\ is the $G$-invariant part of the
 coherent sheaf $\pi_*\Omega^p_Y$. 
 	\end{lemma} 
	
\begin{proof}
This is a special case of \cite[Theorem 2]{MR1451789}.
Note that since $G$ is finite, being horizontal
(see \cite[Introduction]{MR1451789} for the definition)
is an empty condition and $\pi$ contracts no divisor.
\end{proof}
		
\begin{corollary}\label{corollary:differential_forms_on_coarse_moduli}
Let
$
\cX
$
be a smooth Deligne--Mumford stack separated over
$
\bfk
$
of characteristic \(0\),
which admits a coarse moduli scheme
$
 c \colon \cX \to X.
$
Then there exist a natural isomorphism
\begin{equation}
\tilde{\Omega}^p_X
\simto
\bR c_*\Omega^p_{\cX}.
\end{equation}
\end{corollary}

\begin{proof}
The natural morphism
$
 c_*\Omega_{\cX}^p \to \bR c_*\Omega^p_{\cX}
$
is an isomorphism by
\cite[Lemma 2.3.4]{MR1862797}.
In the rest we show that the canonical morphism
$
 c^* \colon \tilde{\Omega}^p_X \to
c_*\Omega^p_{\cX}
$
is an isomorphism.

By
\cite[Lemma 2.2.3]{MR1862797},
\'etale locally on
$
X
$,
the morphism $c$ is of the form
\begin{equation}
c \colon [ U / G ] \to U / G.
\end{equation}
Here $ G $ is a finite group acting algebraically on
a smooth affine variety $ U $.
Recall that there exists a canonical equivalence
of categories
$
 \coh [ U / G ] \simto \coh^G U
$
and that, under this equivalence,
the pushforward functor
$
 c_* \colon \coh [U / G] \to \coh U/G 
$
is identified with
\begin{equation}
\coh^G U \to \coh U/G; \quad \quad F \mapsto (\pi_* F)^G.
\end{equation}
Here
$
\pi
\colon
U
\to
U / G
$
is the quotient morphism. Then we get
\begin{align}
    c _{ \ast } \Omega _{ U } ^{ p }
    =
    \left(
        \pi _{ \ast } \Omega _{ U } ^{ p }
    \right)
    ^{ G }
    \stackrel{\textrm{\cref{lemma:Kahler_differential_under_quotient}}}{\simeq}
    \tilde{ \Omega } _{ U / G } ^{ p }.
\end{align}
\end{proof}

\begin{lemma}\label{lemma:hodge_symmetry}
Let
$
\cX
$
be a smooth proper Deligne--Mumford stack over
$
\bC
$. 
Then the Hodge symmetry
\begin{equation}
\dim H^q ( \cX , \Omega^p_{\cX} )
=
\dim H^p ( \cX , \Omega^q_{\cX} )
\end{equation}
holds.
\end{lemma}

\begin{proof}
Let
$
 c \colon \cX \to X
$
be the morphism to the coarse moduli space.
Since
$
X
$
is a proper $V$-manifold which is bimeromorphic to a K\"{a}hler manifold,
the Hodge symmetry
\begin{equation}
\dim H^q ( X , \tilde{\Omega}^p_{X} )
=
\dim H^p ( X , \tilde{\Omega}^q_{X} )
\end{equation}
holds for any
$
( p , q )
$
by \cite[Theorem 2.43]{MR2393625}.
By
\cref{corollary:differential_forms_on_coarse_moduli} we obtain the isomorphisms
\begin{equation}
H^p ( X , \tilde{\Omega}^q_{X} )
\simeq
H^p ( X , \bR c_*\Omega^q_{\cX} )
\simeq
H^p ( \cX , \Omega^q_{\cX} ),
\end{equation}
to conclude the proof.
\end{proof}

Next we define the notion of orientation sheaf
for normal crossing pairs of Deligne--Mumford stacks; namely, pairs
\(
   \left(
    \cX,
    \cD
   \right)
\)
of a smooth Deligne--Mumford stack
\(
   \cX
\)
and a normal crossing divisor
\(
   \cD
\) on it.
In the case of varieties,
this is originally introduced in \cite[3.1.4]{zbMATH03348292}.

\begin{definition-proposition}\label{definition-proposition:orientation}
Let $(\cX, \cD)$ be a normal crossing pair of a smooth Deligne--Mumford stack
$\cX$ and a normal crossing divisor $\cD$ on it.
For each non-negative integer
$
k
$,
let $\cD^{(k)} \subset \cX$ be the closed substack defined by the following rule;  
\begin{itemize}
\item $\cD^{(0)} = \cX$. 
\item $\cD^{(1)} = \cD$. 	
\item $\cD^{(k)} = \Sing \cD^{(k-1)}$ for $k>1$. 
\end{itemize}

Let $\varphi^k \colon \cD^{k} \to \cD^{(k)}$ be the normalization. 
Take an \'{e}tale atlas
$
U \to \cX
$ by a smooth scheme $U$ such that
$
D\coloneqq \cD \times_{\cX} U
$
is a simple normal crossing divisor. 
Set
$
D^{(k)}\coloneqq \cD^{(k)} \times_{\cX} U
$,
$
D^{k} \coloneqq \cD^{k} \times_{\cX} U
$
and
$
\varphi_U^{k} \colon D^{k} \to U
$
the natural morphism. 
Consider the locally constant sheaf
$E_{D^k}$
on
$D^k$
of finite sets of cardinality $k$ defined as follows:
for a connected open set $W \subset D^{k}$,
let $E_{D^k}(W)$ be the set of irreducible components of
$D$ which contain $\varphi_U^{k}(W)$. 

Let $p_{i} \colon U \times_{\cX} U \to U$ and 
$p^{k}_i \colon D^{k} \times_{\cD^{k}} D^{k} \to D^{k}$ 
be the projections for $i=1,2$. 
We define the gluing isomorphisms
$
\phi^U_{12} \colon (p^{k}_2)^* E_{D^k} \simto (p^{k}_1)^* E_{D^k}
$ as follows. 
Let $E_{D^k \times_{\cD^{k}} D^k}$ be the local system on $D^{k} \times_{\cD^{k}} D^{k}$ corresponding to the atlas
$U \times_{\cX} U \to \cX$.
Since
$
p_i^*D
$
is canonically isomorphic to
$
D \times_{\cD} D
$,
we can define natural isomorphisms
$
b_i \colon
E_{D^k \times_{\cD^k} D^k}
\simto
(p^{k}_i)^* E_{D^k}
$ 
for $i=1,2$. 
Then we set
\[
\phi^U_{12}
\coloneqq
b_1 \circ b_2^{-1}
\colon
(p^{k}_2)^* E_{D^k}
\simto
(p^{k}_1)^* E_{D^k}.
\]
 
On the triple product
$
 D^k \times_{\cD^k} D^k \times_{\cD^k} D^k
$
we can check the cocycle condition
\[
p_{12}^* \phi^U_{12} \circ p_{23}^* \phi^U_{23} = p_{13}^* \phi^U_{13}, 
\] 
where
$
 p_{ i j } \colon D^k \times_{\cD^k} D^k \times_{\cD^k} D^k \to D^k \times_{\cD^k} D^k
$
is the projection to the $i$-th and the $j$-th components for $1 \le i < j \le 3$. 
Thus these data determines a local system $E_{\cD^{k}}$ on $\cD^{k}$. 

Consider the locally free sheaf
$
\cF_{\cD^k}
\coloneqq
\cO_{\cD^k}^{\oplus E_{\cD^k}}
$
of rank
$
k
$
on
$
\cD^k
$.
Concretely, for each connected open subset
$
W \subset D^k
$
we have the description
\[
 \cF_{D^k}(W) = \bigoplus_{\Delta \in E_{D^k}(W)} \cO_{D^k}(W) \cdot e_{\Delta}, 
\]
where $e_{\Delta}$ is a formal basis element corresponding to the irreducible component $\Delta$. 
The gluing isomorphisms
$
\phi^U_{12}
$
of
$
E_{D^k}
$
provides us with the canonical isomorphisms
$
p_2^* \cF_{D^k} \simto p_1^* \cF_{D^k}
$
satisfying the cocycle conditions.

Finally we define the \emph{orientation sheaf} as the line bundle 
$
\cE_{\cD^{k}}\coloneqq \det \cF_{\cD^{k}} 
$ 
on $\cD^{k}$.
One can easily check
$
\cE_{\cD^{k}}^{\otimes 2} \simeq \cO_{\cD^k}
$.
Since the local system
$E_{\cD^1}$
is trivial, we have
$
\cE_{\cD^1} \simeq \cO_{\cD^1}
$.
\end{definition-proposition}

	\begin{lemma}\label{lemma:resolution_for_snc}
	Let
	$
	( \cX , \cD )
	$
	be a normal crossing pair of Deligne--Mumford stacks.
	Then for each
	$p \ge 0$ there exist the following exact sequences.
	\begin{equation}\label{eq:res_for_log_differential_forms}
	0
	\to
	\Omega_{\cX}^p ( \log \cD )( - \cD )
	\to
	\Omega_{\cX}^p
	\to
	\varphi^{1}_* \Omega_{\cD^{1}}^p(\cE_{\cD^{1}}) 
	\xto{r_1} 
	\varphi^{2}_* \Omega_{\cD^{2}}^p (\cE_{\cD^{2}})
	\xto{r_2}
	\cdots
	\end{equation}

	\begin{equation}\label{eq:res_for_log_vector_fields}
	\begin{split}
	0
	\to
	\Theta_{\cX}^p (- \log \cD )
	\to
	\Theta_{\cX}^p
	\to
	\varphi^1_*
	\lb
	\Theta_{\cD^{1}}^{p-1}(\cE_{\cD^1}) \otimes \det \lb \cN_{\cD^{1} / \cX} \rb
	\rb
	\\
	\to
	\varphi^2_*
	\lb
	\Theta_{\cD^{2}}^{p-2}(\cE_{\cD^2}) \otimes \det \lb \cN_{\cD^{2} / \cX} \rb
	\rb
	\to
	\cdots
	\end{split}
	\end{equation}
	\end{lemma}

	\begin{proof}
	For each integer $k$, the restriction homomorphism 
	\[
	r_k \colon \varphi^{k}_* \Omega_{\cD^{k}}^p(\cE_{\cD^{k}}) 
	\to \varphi^{k+1}_* \Omega_{\cD^{k+1}}^p(\cE_{\cD^{k+1}})
	\]
	is defined as follows.
	 Choose an \'{e}tale atlas
	 $
	 U \to \cX
	 $
	by a smooth scheme $U$ on which
	 $D\coloneqq \cD \times_{\cX} U$ is a simple normal crossing divisor. 
	 Denote by
	 $D= \bigcup_{i \in I} D_{i}$ the decomposition of 
$D$ into its irreducible components. 
	 Then
	 $
	 D^{k}\coloneqq \cD^k \times_{\cX} U
	 $
	 can be written as 
	 \[
	 D^{k} = \coprod_{i_1 < \cdots < i_k} D_{ i_1 \dots i_k},  
	 \] 
	 where $D_{ i_1 \dots i_k}\coloneqq D_{ i_1} \cap \cdots \cap D_{ i_{k}}$. 
	 Let $\lambda_{j,k} \colon D_{ i_1 \dots i_{k+1}} \to D_{ i_1\dots i_{j-1} i_{j+1}\dots i_{k+1} }$ 
	 be the natural closed immersion. 
	Then we define  
	\[
	r^U_{k} \colon (\varphi^{k}_U)_* \Omega^i_{D^{k}}(\cE_{D^{k}}) \to 
	(\varphi^{k+1}_U)_* \Omega^i_{D^{k+1}}(\cE_{D^{k+1}}) 
	\]
	as the sum over
	$
	j = 1, 2, \dots, k+1
	$
	of the homomorphisms
	\[
	\lambda_{j,k}^* \otimes (e_{D_{i_j}} \wedge ) \colon 
	\lambda_{j,k}^* \Omega^i_{D_{ i_1\dots i_{j-1} i_{j+1}\dots i_{k+1} }} (\cE_{D_{ i_1\dots i_{j-1} i_{j+1}\dots i_{k+1} }}) 
	\to \Omega^i_{D_{ i_1 \dots i_{k+1} }} (\cE_{D_{ i_1\dots i_{k+1} }}), 
	\]
	where $e_{D_{i_j}} \in \cF_{D^k}$ is the element corresponding to the irreducible component $D_{i_j}$ of $D$. 
	
	We can check that $p_1^* r^U_k = p_2^* r^U_k$ holds for the projections 
	$p_i \colon U \times_{\cX} U \to U$ by the construction of $\cE_{\cD^k}$. 
	Hence we see that the homomorphism $r_k \colon \varphi^{k}_* \Omega_{\cD^{k}}^p(\cE_{\cD^{k}}) 
	\to \varphi^{k+1}_* \Omega_{\cD^{k+1}}^p(\cE_{\cD^{k+1}})$
	is well-defined, so as to obtain the sequence \eqref{eq:res_for_log_differential_forms}. The exactness follows from
	that on the atlas, which is standard
	(cf. \cite[p. 40]{Fujino:2009mz}). 
	
	Finally we obtain \eqref{eq:res_for_log_vector_fields} from \eqref{eq:res_for_log_differential_forms} via the isomorphism
	\(
		\Omega _{ \cX } ^{ n - p } ( \log \cD )
		\simeq
		\Theta _{ \cX } ^{ p } ( - \log \cD ) \otimes \omega _{ X } ( \cD )
	\),
	which follows from the standard	perfect pairing
	\begin{equation}\label{eq:perfect_pairing}
	\Omega^p_{\cX}( \log \cD )
	\times
	\Omega^{n-p}_{\cX}( \log \cD )
	\to
	\omega_{\cX} ( \cD ).
	\end{equation}
	\end{proof}

	\begin{lemma}\label{lemma:spectral_sequence}
	Let
	\begin{equation}
	F^{\bullet} = ( \cdots \to 0 \to F^0 \to F^1 \to \cdots )
	\end{equation}
	be a complex of sheaves on a Deligne--Mumford stack
	$
	\cX
	$.
	Then there exists the following spectral sequence.
	\begin{equation}\label{eq:spectral_sequence}
	E^{ p , q}_1 = H^q ( \cX , F^p )
	\Rightarrow
	E^{p+q} = \bH^{p+q} ( \cX , F^{\bullet} )
	\end{equation}
	\end{lemma}
	
\begin{proof}
This is standard. For example, see \cite[Lemma-Definition A.46]{MR2393625}. 
\end{proof}

We need the following variant of \cite[Theorem 2.1]{MR894379} for a Deligne--Mumford stack. 

\begin{lemma}\label{lemma:log_Frobenius_decomp}
	Let
	$
	( \cX , \cD )
	$
	be a normal crossing pair of Deligne--Mumford stacks over a perfect field
	$\bfk$ of characteristic $p>0$. 
	Let $S\coloneqq \Spec \bfk$ and
	$
	F = F_{X/S} \colon \cX \to \cX'\coloneqq \cX \times_{S, F_S} S
	$
	be the relative Frobenius morphism as explained in
	\cite[Section 1.1]{MR3060750}. 
	Let $W_2(\bfk)$ be the ring of truncated Witt vectors and
	$
	\Stilde\coloneqq \Spec W_2(\bfk)
	$. 
	
\begin{enumerate}[(i)]
\item There is a unique isomorphism of graded $\cO_{\cX'}$-algebras
	\[
	C^{-1} \colon \bigoplus_{i \ge 0} \Omega^i_{\cX'/S}(\log \cD')
	\to
	\bigoplus_{i \ge 0} \cH^i (F_* \Omega^{\bullet}_{\cX /S}(\log \cD) )
	\] 
	such that $C^{-1}(d(x \otimes 1)) = x^{p-1} dx$. 
\item Assume that there exists a lift $(\cXtilde, \cDtilde)$ of $(\cX, \cD)$ over $\Stilde$. 
	Then there is an associated isomorphism 
	\[
	\varphi \colon \bigoplus_{i<p} \Omega^i_{\cX'/S}(\log \cD')
	\to
	\tau_{<p} F_* \Omega^{\bullet}_{\cX/S}(\log \cD)
	\]
	in the derived category of $\cO_{\cX'}$-modules such that $\cH^i(\varphi) = C^{-1}$ for $i <p$. 
\end{enumerate}
\end{lemma}

\begin{proof}
	The first claim is the logarithmic version of	
	\cite[Corollary 2.2]{MR3060750}. The original proof literally works if we replace the Cartier isomorphism with its logarithmic version due to Katz \cite[(7.2.4)]{MR0291177}.
	
	The second claim for schemes is in
	\cite[4.2.3]{MR894379}. 
	With a little more care, as we explain next,
	the proof works for Deligne--Mumford stacks as well.
	Below we use the symbols of \cite[Section 1.1]{MR3060750}.
	
	Consider an \'etale cover
	$
	\cU = \{ U_i \}
	$
	of the stack $\cX$ such that $U_i$ are affine.
	Let $D_i\coloneqq \cD|_{U_i}$ and $(\tilde{U}_i, \tilde{D}_i)$ a lift of $(U_i, D_i)$ over $\Stilde$ 
	induced by $(\cXtilde, \cDtilde)$. 
	Since
	$
	U'_i
	$
	is affine, as explained in
	\cite[Example 10.3]{MR1193913},
	we obtain a lift
	$\Ftilde_i
	\coloneqq
	\Ftilde_{U_i/S}$
	of
	$F_{U_i/S}$
	satisfying
	${\Ftilde_i}^{-1}({\widetilde{D}}'_i)=p \widetilde{D}_i$.
	Thus we obtain the morphisms
	$f_i \colon \Omega^1_{U'_i/S} (\log D'_i) 
	\to \lb F_{U_i/S} \rb_* \Omega^1_{U_i/S}(\log D_i)$. 
	Now, as pointed out in \cite[Remark 1.4]{MR3060750},
	the relative Frobenii commute with \'etale morphisms.
	This ensures that the homotopy maps
	``$
	h_{ij}
	=
	\Ftilde_j^* - \Ftilde_i^*$"
	relating $f_j$ and $f_i$
	are well defined.
	The rest of the construction of $\varphi$ works without change.
	\end{proof}

\begin{lemma}\label{lemma:log_hodge_decomposition}
	Let $
	( \cX , \cD )
	$
	be a normal crossing pair of Deligne--Mumford stacks over $\bC$. 
	Then the spectral sequence 
	\[
	E_1^{p,q}
	\coloneqq
	H^q(\cX, \Omega^p_{\cX}(\log \cD))
	\Rightarrow 
	\bH^{p+q}(\cX, \Omega^{\bullet}_{\cX}(\log \cD))
	\simeq
	H^{p+q} \lb (\cX \setminus \cD)^{\an}, \bC \rb
	\]
	degenerates at $E_1$.  As a consequence, we have an isomorphism
	\begin{equation}
	H^i ( (\cX \setminus \cD)^{\an}, \bC )
	\simeq
	\bigoplus_{q+p = i}
	H^q ( \cX , \Omega_{\cX}^p ( \log \cD ))
	\end{equation}
	for any
	$
	i \ge 0
	$.
\end{lemma}

\begin{proof}
When $\cX$ is a scheme, this is proved in \cite[Corollaire 4.2.4]{MR894379}. 
We can generalize this to the case where $\cX$ is a Deligne--Mumford stack by \cref{lemma:log_Frobenius_decomp} and 
arguments as in \cite[Corollary 1.7]{MR3060750}
and the proof of \cite[Lemma 1.9]{MR3060750}.
\end{proof}

As an application we obtain the stacky and
logarithmic generalization of the
Kodaira--Akizuki--Nakano vanishing theorem.

\begin{definition-proposition}
Let
$
\cX
$
be a Deligne--Mumford stack
which admits a coarse moduli scheme
$
p \colon
\cX \to X
$.
Then for any line bundle
$
\cL
$
on
$
\cX
$,
there exists a positive integer
$
N
$
such that
$
\cL^{\otimes N} \simeq p^* L
$
holds for some line bundle
$
L
$
on
$
X
$.

We say that
$
\cL
$
has a \emph{property
$
\cP
$}
if the descent
$
L
$
as above has the property
$
\cP
$; here
$
\cP
$
is a property of line bundles on schemes
which depends only on the isomorphism classes of
$\bQ$-line bundles, such as ampleness.
\end{definition-proposition}

\begin{theorem}\label{theorem:logarithmic_KAN}
	Let
	$
	\cX
	$
	be a smooth proper Deligne--Mumford stack of dimension
	$d$
	over a perfect field
		$\bfk$ of characteristic $p \ge 0$,
	$
	\cD
	$
	a normal crossing divisor on
	$
	\cX
	$,
	and
	$\cL$ an ample line bundle on
	$
	\cX
	$. 
	Then we have
	\begin{equation}
	 H^i(\cX, \Omega^j_\cX (\log \cD)(-\cD) \otimes \cL) =0 
	\end{equation}
	if $i +j > \bfd(d,p)$,
	where
	\begin{equation}
	\bfd(d,p) \coloneqq
	\begin{cases}
	\sup (d, 2d-p) & p > 0,
	\\
	d & p = 0.
	\end{cases}
	\end{equation}
\end{theorem}

\begin{proof}
	\begin{step}\label{step:no_boundary}
	Let us first consider the case
	$
	\cD = \emptyset
	$.
Since we have \cref{lemma:log_Frobenius_decomp}, we can prove the statement 
by the same arguments as in the proof of \cite[Corollaire 2.8 (i)]{MR894379} when $p >0$.  

We can also prove the statement for $p=0$
by the same arguments as in the proof of \cite[Corollaire 2.11]{MR894379}. 
Note that there exists an integral domain $A$ of finite type over $\bZ$, an injective homomorphism $A \hookrightarrow \bfk$, and 
a smooth proper Deligne--Mumford stack $\cY$ over $A$ which pulls back to $\cX$ over $\bfk$ as explained in \cite[Corollary 1.7]{MR3060750}.

		\end{step}
	
	\begin{step}
	Applying
		\cref{lemma:spectral_sequence}
		to the exact sequence
		\eqref{eq:res_for_log_differential_forms}
		tensored with
		$
		\cL
		$,
		we obtain the spectral sequence
		\[
		E_1^{s,t} = H^t(\cD^{s}, \Omega^j_{\cD^{s}}(\cE_{\cD^s}) \otimes (\varphi^s)^*\cL) \Rightarrow 
		E^{s+t}
		=
		H^{s+t}(\cX, \Omega^j_{\cX}(\log \cD)(-\cD) \otimes \cL).
		\]
		Then for each $(s, t)$ such that $s+t =i$, we have
		$
		E_1^{s,t} =0
		$.
		This follows from \cref{step:no_boundary}, since
		$
		\cE_{\cD^s} \otimes (\varphi^s)^*\cL
		$
		is ample and  
		$
		j+t = j+i-s > \bfd(d,p) - s \ge \bfd(\dim \cD^{s}, p)
		$.
		Thus we obtain
		$
		E^i
		=
		0
		$
		to conclude the proof. 
	\end{step}
	\end{proof}

\section{The classification of twisted sectors of codimension $\le 2$}
\label{section:The_classification_of_twisted_sectors_of_codimension_at_most_2}
In this section we classify twisted sectors
$
 W \subseteq I \scMbar _{ g, n }
$
such that
$
 c_W \le 2
$.
For that purpose, we first start with the classification of those with
$
 F ( W ) \cap \scM _{ g, n } \neq \emptyset
$, where
$
 F \colon I \scMbar _{ g, n } \to \scMbar _{ g, n }
$
is the canonical map. The classification of such
$
 W
$
amounts to that of $n$-pointed smooth projective curves equipped with
nontrivial automorphisms. For each prime number
$
 \ell
$ and positive integer $k$, 
we classify families of smooth projective curves of genus
$
 g \ge 2
$
and (effective) actions of the group
$
 G = \bZ / \ell^k \bZ
$.
Since Pagani has already classified 
the twisted sectors of $\cM_2$ in \cite[p.1652]{MR2871153},
we may and will assume that $g \ge 3$ after the following two remarks. 

\begin{remark}\label{remark:automorphism_of_curves_with_genus_at_most_two}
Here we discuss automorphisms of smooth stable pointed curves of genus $\le 1$.

A stable pointed curve of genus $0$ never has a nontrivial automorphism (even singular ones).

Any elliptic curve, namely, a smooth projective curve $E$ with a marked point $o$, has a unique involution of order two. We will call it the \emph{negation} of $( E, o )$.

There are exactly two isomorphism classes of elliptic curves with extra automorphisms.
One of them is
$
 \lb E _{ 1728 } \coloneqq V \lb z y ^{ 2 } - x ^{ 3 } + x z ^{ 2 } \rb, ( 0 : 0 : 1 ) \rb
$,
with the complex multiplication automorphism of order four given by
$
 \sigma _{ 1728 } \colon \lb x : y : z \rb \mapsto \lb - x : \sqrt{ - 1 } y : z \rb
$.

The other one is the Fermat cubic curve
$
 \lb E _{ 0 } \coloneqq V \lb x ^{ 3 } + y ^{ 3 } + z ^{ 3 } \rb, ( 1 : - 1 : 0 ) \rb
$,
with the complex multiplication automorphism of order six given by
$
 \sigma _{ 0 } \colon \lb x : y : z \rb \mapsto \lb y : x : \omega z \rb
$,
where
$
 \omega
$
is a primitive third root of unity.\end{remark}

\begin{remark}\label{remark:pagani}
Let us recall from \cite[p.1652]{MR2871153} the classification of twisted sectors $W$ of $\scM_2$ such that 
$c_W \le 2$. 

Each $C \in \scM_2$ has the hyperelliptic involution and it gives the twisted sector $W_0$ with $c_{W_0} =0$. 

There is a unique twisted sector $W_1 \subset I \scM_2$ of  $c_{W_1} =1$ whose general member is a double cover of an elliptic curve equipped with the covering involution.

There are  3 twisted sectors of codimension 2.
The first one is $W_2$, whose general member is a triple cover of $\bP^1$ together with a nontrivial covering transformation.
The second one is $W_3$, whose general member is a $\bZ / 4 \bZ$-cover of $\bP^1$ together with a covering transformation of order 4.
The third one is $W _{ 4 }$, whose general member is a
 $
  \bZ / 6 \bZ
 $-cover of $\bP^1$ together with a covering transformation of order 6.

\end{remark}

Consider a smooth projective curve
$
 C
$
of genus
$
 g \ge 3
$
equipped with an effective action
$
 \bZ / \ell ^{ k } \bZ = G \curvearrowright C
$ for a prime number $\ell$.
Let
$
 \pi \colon C \to C / G \eqqcolon \Cbar
$
be the quotient morphism with the branch points $q_1, \ldots, q_r \in \Cbar$
and set
$
 \gbar = g ( \Cbar )
$.
The curve
$
 C
$
together with the
$
 G
$-action is recovered from the isomorphism class of the marked curve
$
 \lb \Cbar, q _{ 1 }, \dots, q _{ r } \rb
$
up to finite choices, since $G$-covers of $\Cbar$
are in one-to-one correspondence with surjective group homomorphisms from
the finitely generated group
$
 \pi _{ 1 } \lb \Cbar \setminus \lc q _{ 1 }, \dots, q _{ r } \rc \rb
$
to the finite group $G$.
Let us classify quintuples
$
 (g, \gbar, r, \ell, k)
$
which appear this way and such that
\begin{equation}\label{eq:codim}
 c ( g, \gbar, r )
 \coloneqq 
 \dim \scM _{ g } - \dim \scM _{ \gbar, r }
 =
 3 ( g - 1 ) - 3 ( \gbar - 1 ) - r
 \le
 2.
\end{equation}

For each
$
 i = 1, 2, \dots, r
$, we have $\pi^* (q_i) = \ell^{k_i} D_i$ for some $0 < k_i \le k$ and 
a reduced divisor $D_i$ on $C$ of degree $\ell^{k-k_i}$.   
By this and the Hurwitz formula, we obtain 
\begin{align}\label{eq:Hurwitz_formula}
 2(g-1) = 2 \ell^k (\gbar -1) + \sum_{i=1}^r
(\ell^k - \ell^{k-k_i}).
\end{align}

By $\ell^{k-k_i} \le \ell^{k-1}$, we  obtain  
\begin{equation}\label{eq:Hurwitzineq}
2(g-1) \ge 2 \ell^k(\gbar-1) + r \ell^{k-1}(\ell-1).  
\end{equation} 
Combining this with \eqref{eq:codim}, we obtain the inequality
\begin{equation}\label{eq:codimineq}
 3 ( \ell ^{ k } - 1 ) ( \gbar - 1 ) + r \lb \tfrac{3}{2} \ell ^{ k - 1 } ( \ell - 1 ) - 1 \rb
 \le
 2.
\end{equation}
It immediately follows that $\gbar \le 1$. Combining this with \eqref{eq:Hurwitz_formula} and the assumption
$
 g \ge 3
$,
it also follows that
$
 r > 0
$.

We shall next show that $k = 1$.
Suppose for a contradiction $k \ge 2$, and
consider first the case $\gbar=1$.
By \eqref{eq:codimineq}, we have 
\begin{align}\label{eq:inequality_gbar_1}
 2 \ge r \lb \tfrac{3}{2}\ell^{k-1}(\ell -1)-1 \rb
\end{align}
and the only possibility is
$
 ( \ell, k ) = ( 2, 2 )
$, in which case $r=1$. 
Then $c(g, \gbar, r) = 3 g - 4 \le 2$, contradicting the assumption
$g \ge 3$.

Next consider the case $\gbar =0$, so that
$
 6 = 3 ( 3 - 1 ) \le 3 ( g - 1 ) \le 3 ( 0 - 1 ) + r + 2
 \Rightarrow
 7 \le r
$
by \eqref{eq:codim}. 
By \eqref{eq:codimineq}, we have 
\[
2 \ge -3(\ell^k-1) + r\lb \tfrac{3}{2}\ell^{k-1}(\ell-1)-1 \rb
= 
3 \ell^{k-1}\lb \lb \tfrac{1}{2} r - 1 \rb ( \ell - 1 ) - 1 \rb - r +3. 
\]
By the above inequality,
$
 r \ge 7
$,
and $\ell, k \ge 2$, we obtain  
\[
 2 \ge 3 \cdot 2 \lb \tfrac{1}{2} r - 1 - 1 \rb - r + 3 = 2 r - 9,
\]
which contradicts
$
 r \ge 7
$.

\vspace{5mm}

Therefore we may and will assume $k=1$. Since
$
 \ell
$
is prime,
$
 \pi
$
is totally ramified in
$
 r \ge 1
$
points on
$
 C
$
with the ramification index
$
 \ell - 1
$.
\eqref{eq:Hurwitz_formula} specializes to
\begin{equation}\label{eq:Hurwitz}
 2 ( g - 1 ) = 2 \ell ( \gbar -1 ) + r ( \ell - 1 ).
\end{equation}
In particular,
$
 r \ge 2
$
is an even integer if $\ell = 2$.

If
$
 \gbar = 1
$,
then we get
$
 c ( g, \gbar, r ) = \lb \tfrac{3}{2} ( \ell - 1 ) - 1 \rb r
$.
We can check that 
$c ( g, \gbar, r ) \le 2$ 
only if
$
 ( \ell, r ) = ( 2, 2 ), ( 2, 4 ), ( 3, 1 )
$.
The cases
$
 (\ell, r) = (2,2)
$
and
$
 ( 3, 1 )
$
are about the cases when $g = 2$, which we have already discussed in \cref{remark:pagani}.
The case
$
(\ell, r) = (2, 4) 
$
is realized by a family of genus $3$ curves which admits an involution
whose quotient is an elliptic curve. The locus of such curves is of codimension $2$
in
$
 \scM _{ 3 }
$; to count the dimension of the locus, note that such a curve with involution corresponds to a genus $1$ curve with 4 marked points up to finite choices.

Finally, the case
$
 \gbar = 0
$
can be settled by similar arguments. 
By \eqref{eq:codim}, 
we get
\begin{align}
 c ( g, \gbar, r ) = \tfrac{1}{2} ( r - 2 ) ( 3 \ell - 5 ) - 2.
\end{align}
As above, the only possibilities are
$(\ell, r)= (2,6), (2,8), (2,10), (3,4)$.
The first three cases correspond to the hyperelliptic loci of
$
 \scM _{ g }
$
for
$
 g = 2, 3, 4
$,
where the codimension of the locus is
$
 0, 1, 2
$,
respectively. 
The case $(\ell, r)=(3,4)$ occurs on $\scM_2$ and corresponds to curves of genus $2$ 
with an automorphism of order $3$. 
Indeed, as a hyperelliptic curve, such a curve is realized as a double cover of $\bP^1$ 
branched at a set of $6$ points which is preserved by an automorphism 
of $\bP^1$ of order $3$. Such an automorphism must freely act on the set of $6$ points, and
lifts to an automorphism of order three of the double cover.
Such configurations of $6$ points on $\bP^1$ constitutes a codimension $2$ 
locus on $\scM_{0,6}$; namely, the locus is of dimension
\(
   1
\). To see this, take a collection of three marked points in one orbit and send them to
\(
   0,
   1
\),
and
\(
   \infty
\)
by one of the six automorphisms of $\bP^1$. The remaining three marked points constitute an orbit, so we are left with one-dimensional freedom to choose them.

\begin{table}[t]
\begin{align}
\begin{array}{cl}
\toprule
c(g, \gbar, r) & (g, \gbar; \ell^k, r)\\
\midrule
0 & 
\\
1 &
(3,0;2,8)\\
2 &
(3,1;2,4), (4,0; 2,10)\\
\bottomrule
\end{array}
\end{align}
\caption{possible values for $(g, \gbar; \ell^k, r)$ with
$
 g \ge 3
$}
\label{table:table_of_twisted_sectors_of_smooth_curves}
\end{table}

\begin{summary}\label{summary:classify_quintuples}
The possible values for $(g, \gbar; \ell^k, r)$ with
$
 g \ge 3
$
are summarized in \cref{table:table_of_twisted_sectors_of_smooth_curves}.
Combining this with
\cref{remark:automorphism_of_curves_with_genus_at_most_two,remark:pagani}, we obtain the classification of twisted sectors $W$ with
$
 c _{ W } \le 2
$
of
$
 \scM _{ g, n }
$
for any
$
 \lb g, n \rb
$.
They are listed in
\cref{table:table_of_twisted_sectors_of_smooth_pointed_curves_codim0},
\cref{table:table_of_twisted_sectors_of_smooth_pointed_curves_codim1}, and
\cref{table:table_of_twisted_sectors_of_smooth_pointed_curves_codim2},
according to the value of $c_W$.
\end{summary}

\begin{table}[h]
\begin{align}
\begin{array}{cc}
\toprule
 ( g, n ) & \mbox{pointed curve, automorphism}\\
\midrule
( 2, 0 ) & \mbox{genus 2 curve, HE involution}\\
( 1, 1 ) & \mbox{elliptic curve, negation}\\
\bottomrule
\end{array}
\end{align}
\caption{$
 c _{ W } = 0
$
}
    \label{table:table_of_twisted_sectors_of_smooth_pointed_curves_codim0}
\end{table}

\begin{table}[h]
\begin{align}
\begin{array}{cc}
\toprule
 ( g, n ) & \mbox{pointed curve, automorphism}\\
\midrule
( 3, 0 ) & \mbox{HE genus 3 curve, HE involution}\\
( 2, 1 ) & \mbox{genus 2 curve with a Weierstrass marked point, HE involution}\\
( 2, 0 ) & \mbox{bielliptic genus 2 curve, covering involution}\\
( 1, 2 ) & \mbox{elliptic curve with a 2-torsion marked point, negation}\\
( 1, 1 ) & \mbox{$E _{ 1728 }$ or $E _{ 0 }$, complex multiplication}\\
\bottomrule
\end{array}
\end{align}
\caption{$c_W =1$}
    \label{table:table_of_twisted_sectors_of_smooth_pointed_curves_codim1}
\end{table}

\begin{table}[h]
\begin{align}
\begin{array}{cc}
\toprule
 ( g, n ) & \mbox{pointed curve, automorphism}\\
\midrule
( 4, 0 ) & \mbox{HE genus 4 curve, HE involution}\\
( 3, 1 ) & \mbox{HE genus 3 curve with a Weierstrass marked point, HE involution}\\
( 3, 0 ) & \mbox{bielliptic genus 3 curve, covering involution}\\
( 2, 2 ) & \mbox{genus 2 curve with 2 Weierstrass marked points, involution}\\
( 2, 1 ) & \mbox{bielliptic genus 2 curve with a ramification marked point, covering involution}\\
( 2, 0 ) & \mbox{$\bZ / 3 \bZ$ cover of $\bP ^{ 1 }$, covering transformation}\\
( 2, 0 ) & \mbox{$\bZ / 4 \bZ$ cover of $\bP ^{ 1 }$, covering transformation of order 4}\\
( 2, 0 ) & \mbox{$\bZ / 6 \bZ$ cover of $\bP ^{ 1 }$, covering transformation of order 6}\\
( 1, 3 ) & \mbox{elliptic curve with two 2-torsion marked points, negation}\\
( 1, 2 ) & \mbox{$E _{ 1728 }$ or $E _{ 0 }$ with a fixed marked point, complex multiplication}\\
\bottomrule
\end{array}
\end{align}
\caption{$c_W=2$}
    \label{table:table_of_twisted_sectors_of_smooth_pointed_curves_codim2}
\end{table}

Let us begin the classification of twisted sectors $W$ with
$ c _{ W } \le 2
$
of
$
 I \scMbar _{ g, n }
$.
The case
\(
   c _{ W } = 0
\)
is immediate:
\begin{proposition}\label{proposition:classification_of_twisted_sectors_of_codim_0}
Twisted sectors
$
 W \subset I \scMbar _{ g, n }
$
with
$
 c_W = 0
$
are classified as follows.

\begin{enumerate}
\item
$
 ( g, n ) = ( 2, 0 )
$
and
$
 W
$
represents the pairs consisting of a stable genus 2 curve and its hyperelliptic involution.
In this case
$
 F \colon W \to \scMbar _{ 2, 0 }
$
is an isomorphism.

\item
$
 ( g, n ) = ( 1, 1 )
$
and
$
 W
$
represents the pairs consisting of a stable pointed genus 1 curve and the negation with respect to the group structure whose origin is the marked point.
In this case
$
 F \colon W \to \scMbar _{ 1, 1 }
$
is an isomorphism.
\end{enumerate}
\end{proposition}

Next we classify twisted sectors $W$ with
$ c _{ W } = 1
$.
The key point is the following:
\begin{proposition}\label{proposition:study_of_W_of_codim_1}
If a twisted sector
$
 W \subset I \scMbar _{ g, n }
$
of
$
 c _{ W } = 1
$
satisfies
$
 F ( W ) \cap \scM _{ g, n } = \emptyset
$,
then
$
 F ( W ) = \cB _{ 1 : \emptyset }
$.
\end{proposition}

\begin{proof}
There are four possibilities for
$
 W
$
as follows.
See \cref{definition:boundarydivisors} for the definition of
$
 \cB _{ j : S }
$
and so on. We have the following 3 cases: 

\begin{enumerate}[(i)]
\item
$
 F \lb W \rb
$
coincides with the boundary divisor
$
 \cB _{ j : S }
$, and the automorphisms represented by
$
 W
$
do not exchange the two irreducible components of the curves.

\item
$
 F \lb W \rb
$
coincides with the boundary divisor
$
 \cB _{ j : S }
$, and the automorphisms represented by
$
 W
$
exchange the irreducible components of the curves.

\item
$
 F \lb W \rb
$
coincides with the boundary divisor
$
 \cB _{ 0 }
$.

\end{enumerate}

We investigate each case separately.

\begin{enumerate}[(i)]
\item
Since the normalization morphism
$
 \nu \colon \cB _{ j: S } ^{ \nu } \to \cB _{ j: S }
$
is birational,
$
 W
$
is dominated by a connected component
$
 W ' \subset I \lb \cB _{ j: S } ^{ \nu } \rb
$
of codimension 0 under the morphism
$
 I \nu
$.

Recall that there is an \'etale finite morphism
$
 \rho \colon \scMbar _{ j, S \cup \lc n + 1 \rc } \times \scMbar _{ g - j, \lb [ n ] \setminus S \rb \cup \lc n + 2 \rc }
 \to \cB _{ j: S } ^{ \nu }
$,
which in fact is an isomorphism except when
$
 2 j = g
$
and
$
 n = 0
$, in which case the degree of the morphism is 2.
Since we consider the case when the general point of
$
 W
$
(and hence $ W' $) represents an automorphism which does not exchange the irreducible components of the curve,
it is dominated by a connected component $\Wtilde$ of codimension $0$ of
$
 I \lb \scMbar _{ j, S \cup \lc n + 1 \rc } \times \scMbar _{ g - j, \lb [ n ] \setminus S \rb \cup \lc n + 2 \rc } \rb
$.
By \cref{lemma:inertia_commutes_with_product},
$\Wtilde$ is a connected component of codimension $0$ of the stack
$
 I \scMbar _{ j, S \cup \lc n + 1 \rc }
 \times
 I \scMbar _{ g - j, \lb [ n ] \setminus S \rb \cup \lc n + 2 \rc }
 \setminus
 \scMbar _{ j, S \cup \lc n + 1 \rc }
 \times
 \scMbar _{ g - j, \lb [ n ] \setminus S \rb \cup \lc n + 2 \rc }
$.

As we saw in \cref{proposition:classification_of_twisted_sectors_of_codim_0},
$
 ( g, n ) = ( 1, 1 )
$
is the only pair with
$
 n \ge 1
$
such that
$
 \scMbar _{ g, n }
$
admits a twisted sector of codimension 0.
Thus we see that the general member of
$
 W
$
is a union of a smooth elliptic curve $C '$ and a $n$-pointed smooth curve $C ''$ of genus $g - 1$ meeting in a point, equipped with the negation of $C '$ (glued with the identity automorphism of $C ''$). Note that when
$
 ( g, n ) = ( 2, 0 )
$
we have a family of curves
$
 C ' \cup C ''
$,
where
$
 C '
$
and
$
 C ''
$
are curves of genus $1$, equipped with the negations on both of the components. This does not contribute a codimension $1$ twisted sector of
$
 \scMbar _{ 2, 0 }
$,
since the family is the limit of the family of smooth genus $2$ curves equipped with the HE involutions.

\item
This case never occurs (under the assumption
$
 c _{ W } = 1
$). In fact this can happen only when $n = 0$ and
$
 g = 2 j
$
for some integer $j$, but the codimension of the diagonal
\begin{align}
 \scMbar _{ j, 1 } \hookrightarrow \scMbar _{ j, 1 } \times \scMbar _{ j, 1 }
\end{align}
is always strictly positive for any
$
 j
$
such that
$
 \scMbar _{ j, 1 }
$
is a Deligne--Mumford stack.

\item
We prove by contradiction that this case never occurs.
If
$
 \cB _{ 0 } = F ( W )
$
holds for a twisted sector $W$, then any member of
$
 \cB _{ 0 }
$
admits a nontrivial automorphism.
Taking the normalization of the members of
$
 \cB _{ 0 }
$,
we see that any $n$-pointed smooth curve of genus
$
 g - 1
$
and a pair of points on it, there exists an automorphism of the curve which swaps the pair of points.
Hence the only possibility for
$
 ( g, n )
$
are
$
 ( 1, 1 ), ( 2, 0 )
$.
In each of these cases, one can check that the nontrivial automorphisms of the general member of
$
 \cB _{ 0 }
$
is a degeneration of the HE involutions of the smooth members, thereby obtaining a contradiction.
\end{enumerate}
\end{proof}

\begin{summary}
\label{summary:classification_of_twisted_sectors_of_codim_1}
Twisted sectors
$
 W \subset I \scMbar _{ g, n }
$
with
$
 c_W = 1
$
are classified as follows.

There is a twisted sector of common type whose generic member is the union of a smooth elliptic curve $C '$ and
a $n$-pointed smooth curve $C ''$ of genus $g-1$ meeting in a point, equipped with the negation of $C '$ (glued with the identity automorphism of $C ''$).

\begin{enumerate}
 \item\label{it:classification_of_twisted_sectors_of_codim_2_generic}
 If either
 $
  g \ge 4
 $,
 $
  g = 3
 $
 and
 $
  n \ge 1
 $,
 $
  g = 2
 $
 and
 $
  n \ge 2
 $, or
 $
  g = 1
 $
 and
 $
  n \ge 3
 $, then there is no twisted sector other than the common one.

 \item\label{it:classification_of_twisted_sectors_of_codim_1_(3,0)}
 For
 $
  ( g, n ) = ( 3, 0 )
 $,
 other than the common one, there exists one more twisted sector
 \(
    W _{ \mathrm{HE} }
 \)
  whose general member is a HE curve equipped with the HE involution.

 \item\label{it:classification_of_twisted_sectors_of_codim_1_(2,1)}
 For
 $
  ( g, n ) = ( 2, 1 )
 $,
 other than the common one, there exists one more twisted sector
 \(
    W _{ \mathrm{HE} }
 \)
  whose general member is a curve marked by a Weierstrass point equipped with the HE involution.

 \item\label{it:classification_of_twisted_sectors_of_codim_1_(2,0)}
 For
 $
  ( g, n ) = ( 2, 0 )
 $,
 other than the common one, there exists one more twisted sector
 \(
    W _{ \mathrm{BE} }
 \)
  whose general member is a bielliptic curve equipped with the covering involution.

 \item\label{it:classification_of_twisted_sectors_of_codim_1_(1,2)}
  For
 $
  ( g, n ) = ( 1, 2 )
 $,
 other than
 $
  \cB _{ 1 : \emptyset }
 $, 
 there exists one more twisted sector
 $
  W _{ \mathrm{NE} }
 $
 such that
 $
  W _{ \mathrm{NE} } \cap F ^{ - 1 } \lb \scM _{ 1, 2 } \rb
 $
 classifies a pointed curve
 $
  ( C, o, p )
 $
 such that
 $
  p
 $
 is a
 $
  2
 $-torsion point with respect to the origin
 $
  o
 $, equipped with the involution with respect to $o$
 (cf. \cite[Corollary 3.14]{MR3137360}).

 \item\label{it:classification_of_twisted_sectors_of_codim_1_(1,1)}
  For
 $
  ( g, n ) = ( 1, 1 )
 $,
 there exist $ 2 \cdot ( 1 + 2 ) $ twisted sectors
 \(
    W _{ \mathrm{CM} }
 \)
  of codimension
 $
  1
 $
 corresponding respectively to the CM curves
$
 E _{ 1728 }
$
and
$
 E _{ 0 }
$
equipped with CM automorphisms (cf. \cite[Corollary 3.14]{MR3137360}). Under the well-known isomorphisms
\begin{align}
  \scMbar _{ 1, 1 } \simeq \bP \lb 4, 6 \rb \simeq \bP \lb 2, 3 \rb \times
  \bB \mu _{ 2 },
\end{align}
these twisted sectors arise from the two stacky points of
$
 \bP \lb 2, 3 \rb
$.

 \end{enumerate}
\end{summary}

\cref{table:classification_of_twisted_sectors_of_codim_1} below may be of use.

\begin{table}[h]
\begin{align}
\begin{array}{cccccc}
 \toprule
 g \backslash n & 0 & 1 & 2 & 3 & \ge 4\\
 \midrule
 1 & NA & \eqref{it:classification_of_twisted_sectors_of_codim_1_(1,1)}
 & \eqref{it:classification_of_twisted_sectors_of_codim_1_(1,2)}
 & \eqref{it:classification_of_twisted_sectors_of_codim_2_generic}
 & \eqref{it:classification_of_twisted_sectors_of_codim_2_generic}\\
 2 & \eqref{it:classification_of_twisted_sectors_of_codim_1_(2,0)}
 & \eqref{it:classification_of_twisted_sectors_of_codim_1_(2,1)}
 & \eqref{it:classification_of_twisted_sectors_of_codim_2_generic}
 & \eqref{it:classification_of_twisted_sectors_of_codim_2_generic}
 & \eqref{it:classification_of_twisted_sectors_of_codim_2_generic}\\
 3 & \eqref{it:classification_of_twisted_sectors_of_codim_1_(3,0)}
 & \eqref{it:classification_of_twisted_sectors_of_codim_2_generic}
 & \eqref{it:classification_of_twisted_sectors_of_codim_2_generic}
 & \eqref{it:classification_of_twisted_sectors_of_codim_2_generic}
 & \eqref{it:classification_of_twisted_sectors_of_codim_2_generic}\\
 \ge 4
 & \eqref{it:classification_of_twisted_sectors_of_codim_2_generic}
 & \eqref{it:classification_of_twisted_sectors_of_codim_2_generic}
 & \eqref{it:classification_of_twisted_sectors_of_codim_2_generic}
 & \eqref{it:classification_of_twisted_sectors_of_codim_2_generic}
 & \eqref{it:classification_of_twisted_sectors_of_codim_2_generic} \\
 \bottomrule
\end{array}
\end{align}
\caption{Classification of twisted sectors of codimension 1}
\label{table:classification_of_twisted_sectors_of_codim_1}
\end{table}

\begin{proof}
This immediately follows from
\cref{proposition:study_of_W_of_codim_1}
and
\cref{summary:classify_quintuples}
(in particular, \cref{table:table_of_twisted_sectors_of_smooth_pointed_curves_codim1}).
\end{proof}

Below is a summary of what we have done so far, which will be useful for the classification of twisted sectors of codimension $2$:
\begin{summary}\label{summary:for_the_classification_of_tw_sectors_of_codim_2}
Fix an integer
$
 j \ge 0
$
and a finite set
$
 S
$
such that
$
 \scMbar _{ j, S \cup \lc \bullet \rc }
$
is not empty. Then the connected components
$
 \Wtilde \subset I \scMbar _{ j, S \cup \lc \bullet \rc }
$
of codimension
$
 \le 1
$
are classified as in \cref{table:table_of_twisted_sectors_of_pointed_curves_codim0,table:table_of_twisted_sectors_of_pointed_curves_codim1}. 

\begin{table}[h]
\begin{align}
\begin{array}{cccccc}
\toprule
( j, S )
& \mbox{description of the general member}\\
\midrule
\mbox{general}
& \mbox{untwisted sector}\\
( 1, \emptyset )
& \mbox{elliptic curve with the origin $\bullet$, negation}\\
\bottomrule
\end{array}
\end{align}
\caption{Classification of twisted sectors of codimension 0}
    \label{table:table_of_twisted_sectors_of_pointed_curves_codim0}
\end{table}

\begin{table}[h]
\begin{align}
\begin{array}{cc}
\toprule
( j, S )
& \mbox{description of the general member}\\
\midrule
\mbox{general}
& \mbox{$C \in \cB _{ 1: \emptyset }$, negation of the elliptic component}\\
( 2, \emptyset )
& \mbox{smooth curve such that $\bullet$ is a Weierstrass point, HE involution}\\
( 1, | S | = 1 )
& \mbox{elliptic curve with a 2-torsion marked point, negation}\\
( 1, \emptyset )
& \mbox{$ E _{ 1728 }$ or $E _{ 0 }$, CM automorphism}\\
\bottomrule
\end{array}
\end{align}
\caption{Classification of twisted sectors of codimension 1}
    \label{table:table_of_twisted_sectors_of_pointed_curves_codim1}
\end{table}
\end{summary}

Now we classify twisted sectors of codimension 2. Below is the key point for the classification:
\begin{proposition}\label{proposition:study_of_W_of_codim_2}
A twisted sector
$
 W \subset I \scMbar _{ g, n }
$
of
$
 c _{ W } = 2
$
with
$
 F ( W ) \cap \scM _{ g, n } = \emptyset
$
admits a finite surjective morphism from a twisted sector
$
 \Wtilde \subset
 I \lb \scMbar _{ j, S \cup \lc n + 1 \rc } \times \scMbar _{ g - j, \lb [ n ] \setminus S \rb \cup \lc n + 2 \rc } \rb
$
of
$
 c _{ \Wtilde } = 1
$
for some
\(
   0 \le j \le g
\)
and a subset
\(
   S
   \subseteq
   [ n ]
\).
\end{proposition}

\begin{proof}
As in the case
$
 c _{ W } = 1
$,
there are three possibilities for
$
 W
$
as follows.

\begin{enumerate}[(i)]
\item
$
 F \lb W \rb
$
is contained in the boundary divisor
$
 \cB _{ j : S }
$, and the general point of $W$ represents a pointed curve $C$ with an automorphism which does not exchange the irreducible components of
$
 C
$.

\item
$
 F \lb W \rb
$
is contained in the boundary divisor
$
 \cB _{ j : S }
$, and the general point of $W$ represents a pointed curve $C$ with an automorphism which exchanges the irreducible components of
$
 C
$.

\item
$
 F \lb W \rb
$
is not contained in the boundary divisor of the form
$
 \cB _{ j ; S }
$, but is contained in the boundary divisor
$
 \cB _{ 0 }
$.

\end{enumerate}

We investigate each case separately.
\begin{enumerate}[(i)]
\item
By a modular interpretation, one can easily verify that $F(W)$ is dominated by an irreducible closed substack
$
 W ' \subset \cB _{ j: S } ^{ \nu }
$
of codimension $1$, even if
$
 F ( W ) \subset \Sing \cB _{ j: S }
$.
Recall from \cref{lemma:inductive_structure_of_the_boundary} that there is an \'etale finite morphism
\[
 \rho \colon \scMbar _{ j, S \cup \lc n + 1 \rc } \times \scMbar _{ g - j, \lb [ n ] \setminus S \rb \cup \lc n + 2 \rc }
 \to \cB _{ j: S } ^{ \nu }.
 \]
Since we consider the case when the general point of
$
 W
$
(and hence $W '$) represents an automorphism which does not exchange the irreducible components of the curve,
it follows that $W '$ is dominated by an irreducible closed substack
\[
 \Wtilde \subset I \lb \scMbar _{ j, S \cup \lc n + 1 \rc } \times \scMbar _{ g - j, \lb [ n ] \setminus S \rb \cup \lc n + 2 \rc} \rb
\]
of
$
 c _{ \Wtilde } = 1
$.
If
$
 \Wtilde
$
is not a twisted sector,
then it must be contained in a codimension 0 twisted sector (it can not be the untwisted sector, since the general point of $W$ represents a curve with a nontrivial automorphism). By a modular interpretation, this would lead to a contradiction since it would imply that $W$ is contained in a bigger twisted sector. Thus we see that
$
 \Wtilde
$
is a twisted sector of codimension $1$, whose classification is reduced to the that of the connected components of codimension $\le 1$ of
$
 I \scMbar _{ j, S \cup \lc n + 1 \rc }
$
and
$
 I \scMbar _{ g - j, \lb [ n ] \setminus S \rb \cup \lc n + 2 \rc }
$.

\item
This case does not occur as we discuss below.
Suppose that the general member of
$
 W
$
has two irreducible components. This can happen only when $n = 0$ and
$
 g = 2 j
$.
Since the codimension of the diagonal
\begin{align}
 \scMbar _{ j, 1 } \hookrightarrow \scMbar _{ j, 1 } \times \scMbar _{ j, 1 }
\end{align}
is
$
 3 j - 2
$,
the only possibility (under the assumption
$
 c _{ W } = 2
$)
is
$
 j = 1 \iff g = 2$.
This is the case when the generic member of
$
 W
$
represents a genus 2 curve with 2 isomorphic irreducible components equipped with the involution which identifies the two components.
However, such curves with involutions admit deformation to smooth bielliptic curves with involutions, contradicting the assumption
$
 c _{ W } = 2
$. In fact, such a curve is isomorphic to the double cover of a smooth curve $C$ of genus 1 ramified in a point $ o \in C$. This is realized as the cyclic double cover of $C$ obtained from the line bundle
$
 L = \cO _{ C } ( o )
$
and the global section of
$
 L ^{ \otimes 2 }
$
which corresponds to the divisor
$
 2 o
$.
The desired smoothing is obtained by deforming the global section to the ones with only (two) simple zeros.

There is also the possibility that the general member of $W$ consists of three irreducible components, but in this case
$
 F ( W )
$
must coincide with a codimension 2 stratum and one can easily conclude that this case never occurs (under the assumption
$
 c _{ W } = 2
$); in fact, the existence of the automorphism implies that 2 of the 3 irreducible components of the general member of $W$ should be isomorphic to each other, but 
the general member of the stratum do not satisfy this condition. Hence this case never occurs.

\item
We prove that this case never occurs, by a case-by-case analysis.

Suppose that
$
 F ( W )
$
coincides with an irreducible component of
$
 \Sing \cB _{ 0 }
$.
Consider first the case when the general member of $W$ is an irreducible curve $C$ with 2 nodes equipped with $n$ marked points and a nontrivial automorphism. We prove that this case never occurs.

Taking the normalization of $C$, we obtain a family of smooth $n$ pointed curves of genus $g - 2$ equipped with nontrivial automorphisms. An easy dimension computation implies that either the resulting pointed curves are not stable or they constitute a codimension 0 twisted sector of the moduli stack
$
 \scM _{ g - 2, n }
$.
Moreover, if one takes two pairs of points on such a pointed curve, there must be an automorphism which swaps the pairs.
Thus we see that the only possibility is
$
 ( g, n )
 =
 ( 2, 0 )
$. However, such involutions of nodal curves are degenerations of the smooth genus 2 curves with HE involutions. This contradicts the assumption that
$
 W
$
is a connected component of the inertia stack.

Next we consider the case when the general member of $W$ has two irreducible components. Let the genera of the two components be
$
 g '
$
and
$
 g ''
$, respectively, and
$
 S ', S ''
$
be the set of marked points on the respective components (so that
$
 [ n ] = S ' \coprod S ''
$). A moment's reflection will convince one that the assumption
$
 c _{ W } = 2
$
implies that the automorphism of the general member does not exchange the irreducible components.
An easy consideration implies that this occurs if and only if each of
$
 \lb g ', n ' \rb
$
and
$
 \lb g '', n '' \rb
$
is either
$
 ( 0, 1 )
$
or
$
 ( 1, 0 )
$.
In any of the four cases, the automorphism extends to HE involutions of smooth curves; therefore, this does not yield a new twisted sector.

Now suppose that
$
 F ( W ) \not\subset \Sing \cB _{ 0 }
$, so that the general member of $W$ is an irreducible curve $C$ with 1 node equipped with $n$ marked points and a nontrivial automorphism. We prove that this case does not occur.

Taking the normalization
$
 \Ctilde
$
of $C$, we obtain a family of smooth $n$ pointed curves of genus $g - 1$ equipped with nontrivial automorphisms. An easy dimension computation again implies that either the resulting pointed curve is not stable or this yields a codimension 0 twisted sector of the moduli stack
$
 \scM _{ g - 1, n }
$.
Thus we see that the only possibilities are
$
 ( g, n )
 =
 ( 1, 1 ), ( 1, 2 ), ( 2, 0 ), ( 2, 1 ), ( 3, 0 )
$.
The case
$
 ( 1, 1 )
$
is excluded by the dimension reason.
The case
$
 ( 1, 2 )
$
is excluded by the following argument. Let
$
 ( C, p, q )
$
be a general member of
$
 F ( W )
$.
Then
$
 \lb \Ctilde, p, q \rb
$
is equipped with an automorphism $\iota$ which fixes $p$ and $q$, and preserves the set of 2 points lying over the node. Hence
$
 \iota ^{ 2 } = \id _{ \Ctilde }
$,
and it follows that
$
 ( C, p, q )
$
with the involution admit smoothing to smooth genus 1 curves equipped with the negation with respect to one of the 2 marked points such that the other marked point is a 2-torsion point. This contradicts that
$
 W
$
is a connected component of the inertia stack.

The cases
$
 ( 2, 1 ), ( 3, 0 )
$
are excluded, since as in the case
$
 ( 1, 2 )
$
the only possibility in these cases is a family of nodal curves with involution, and they are degenerations of the family of smooth (pointed) curves with the negation and the hyperelliptic involution, respectively.

Finally, the case
$
 ( 2, 0 )
$
is also excluded by the similar reasoning. In fact if
$
 C
$
is the general member of
$
 F ( W )
$
equipped with an automorphism, one can check that it should actually lift to the negation of the normalization
$
 \Ctilde
$
(with respect to an origin).
Then it follows that the involution on
$
 C
$
is nothing but the degeneration of the hyperelliptic involution, which is a contradiction.
\end{enumerate}
\end{proof}

\begin{summary}
\label{summary:classification_of_twisted_sectors_of_codim_2}
Twisted sectors
$
 W \subset I \scMbar _{ g, n }
$
of
$
 c_W = 2
$
are classified as follows.

Note first that there are 4 types of twisted sectors which are of common nature as follows.

\begin{enumerate}[(I)]
 \item\label{it:classification_of_twisted_sectors_of_codim_2_HE}
 A twisted sector
 $
  W _{ \mathrm{I} }
 $
 whose general point represents a smooth curve of genus
 $
  2
 $
 meeting in a Weierstrass point with another smooth projective
 $
  n
 $-pointed curve of genus
 $
  g - 2
 $,
 equipped with the HE involution (glued with the identity automorphism of the other component).

 \item\label{it:classification_of_twisted_sectors_of_codim_2_Elliptic}
  A twisted sector
 $
  W _{ \mathrm{II} }
 $
 whose general point represents a pointed smooth curve of genus
 $
  1
 $
 meeting in a Weierstrass point
 with another smooth projective
 $
  ( n - 1 )
 $-pointed curve of genus
 $
  g - 1
 $,
 equipped with the negation with respect to the marked point (glued with the identity automorphism of the other component).

 \item\label{it:classification_of_twisted_sectors_of_codim_2_CM}
  A twisted sector
 $
  W _{ \mathrm{III} }
 $
 whose general point represents either
 $
  E _{ 1728 }
 $
 or
 $
  E _{ 0 }
 $
meeting in a point with a $n$-pointed smooth projective curve of genus $g-1$, equipped with the CM automorphism (glued with the identity automorphism of the other component).

 \item\label{it:classification_of_twisted_sectors_of_codim_2_3_components}
 A twisted sector
 $
  W _{ \mathrm{IV} }
 $
 whose general point represents a stable curve of the form $\Gamma = E_1 \cup C \cup E_2$,
 where $E_1$ and $E_2$ are smooth genus 1 curves glued to a smooth curve $C$ of genus $g-2$ with $n$-marked points,
 equipped with the simultaneous negation of the elliptic components (identity on the component $C$).
\end{enumerate}

The exact classification of the twisted sectors of codimension $2$ is as follows.

\begin{enumerate}
 \item\label{it:classification_of_twisted_sectors_of_codim_2_generic}
 Assume either
 $
  g \ge 5
 $,
 $
  g \ge 4
 $
 and
 $
  n \ge 1
 $,
 $
  g = 3
 $
 and
 $
  n \ge 2
 $,
 $
  g = 2
 $
 and
 $
  n \ge 3
 $, or
 $
  g = 1
 $
 and
 $
  n \ge 4
 $.
 Then there is no twisted sector of
 $
  c _{ W } = 2
 $
 other than \eqref{it:classification_of_twisted_sectors_of_codim_2_HE},\eqref{it:classification_of_twisted_sectors_of_codim_2_Elliptic}, \eqref{it:classification_of_twisted_sectors_of_codim_2_CM} and \eqref{it:classification_of_twisted_sectors_of_codim_2_3_components}.

 \item\label{it:classification_of_twisted_sectors_of_codim_2_4_0}
 For
 $
  ( g, n ) = ( 4, 0 )
 $,
 there exists one more twisted sector
 $
  W _{ \mathrm{HE} }
 $
 of
 $
  c _{ W _{ \mathrm{HE} } } = 2
 $
 such that a member of
 $
  W _{ \mathrm{HE} } \cap F ^{ - 1 } \lb \scM _{ 4, 0 } \rb
 $
 is a hyperelliptic curve equipped with the hyperelliptic involution.

 \item\label{it:classification_of_twisted_sectors_of_codim_2_3_1}
 For
 $
  ( g, n ) = ( 3, 1 )
 $,
 there exists one more twisted sector
 $
  W _{ \mathrm{HE} }
 $
 of
 $
  c _{ W _{ \mathrm{HE} } } = 2
 $
 such that a general point of
 $
  W _{ \mathrm{HE} } \cap F ^{ - 1 } \lb \scM _{ 3, 1 } \rb
 $
 represents a hyperelliptic curve marked by a Weierstrass point, equipped with the hyperelliptic involution.

 \item\label{it:classification_of_twisted_sectors_of_codim_2_3_0}
 For
 $
  ( g, n ) = ( 3, 0 )
 $,
 there exists one more twisted sector
 $
  W _{ \mathrm{BE} }
 $
 of
 $
  c _{ W _{ \mathrm{BE} } } = 2
 $
 other than
 \eqref{it:classification_of_twisted_sectors_of_codim_2_HE},
 whose general member is a double cover of an elliptic curve
 branched in two points, equipped with the covering involution. 

 \item\label{it:classification_of_twisted_sectors_of_codim_2_2_2}
 For
 $
  ( g, n ) = ( 2, 2 )
 $,
 there exists another twisted sector
 $
  W _{ \mathrm{HE} }
 $
 of
 $
  c _{ W _{ \mathrm{HE} } } = 2
 $
 such that a member of
 $
  W _{ \mathrm{HE} } \cap F ^{ - 1 } \lb \scM _{ 2, 2 } \rb
 $
 is represented by a
 $
  2
 $-pointed hyperelliptic curve such that both of the marked points are
 Weierstrass points.

 \item\label{it:classification_of_twisted_sectors_of_codim_2_2_1}
 For
 $
  ( g, n ) = ( 2, 1 )
 $,
 we have \eqref{it:classification_of_twisted_sectors_of_codim_2_Elliptic}, \eqref{it:classification_of_twisted_sectors_of_codim_2_CM}, \eqref{it:classification_of_twisted_sectors_of_codim_2_3_components} and one more twisted sector
 $
  W _{ \mathrm{BE} }
 $
 of
 $
  c _{ W _{ \mathrm{BE} } } = 2
 $
 such that a member of
 $
  W _{ \mathrm{BE} } \cap F ^{ - 1 } \lb \scM _{ 2, 1 } \rb
 $
 is represented by a double cover of an elliptic curve branched in
 $
  2
 $-points marked by a ramification point, equipped with the covering involution.

 \item\label{it:classification_of_twisted_sectors_of_codim_2_2_0}
 For
 $
  ( g, n ) = ( 2, 0 )
 $,
 there are
 $
  4
 $ types of twisted sectors
 $
  W
 $
 of
 $
  c_W = 2
 $
 other than \eqref{it:classification_of_twisted_sectors_of_codim_2_Elliptic} and \eqref{it:classification_of_twisted_sectors_of_codim_2_CM}
 as follows (cf. \cite[p.63, 64]{pagani2009chen}):
\begin{enumerate}
\item
 A twisted sector
 \(
    W _{ 3 }
 \)
 whose general point represents a triple cover of $\bP^1$ branched at $4$ points, equipped with a
 covering transformation.

\item
 A twisted sector
 \(
    W _{ 4 }
 \)
 whose general point represents a $\bZ / 4$-cover of $\bP^1$ branched at $4$ points, equipped with
 a covering transformation.
 
\item
 A twisted sector
 \(
    W _{ 6 }
 \)
 whose general point represents a $\bZ / 6$-cover of $\bP^1$, equipped with
 a covering transformation.

 \item
 A twisted sector
 \(
    W _{ \mathrm{CM} }
 \)
 whose general point represents a union
 $
  C ' \cup C ''
 $,
 where $C '$ is a CM curve and $C ''$ is a genus one curve, equipped with the CM automorphism and the negation.
 \end{enumerate}

 \item\label{it:classification_of_twisted_sectors_of_codim_2_1_3}
 For
 $
  ( g, n ) = ( 1, 3 )
 $,
 there is another twisted sector
 $
  W _{ \mathrm{NE} }
 $
 of
 $
  c _{ W _{ \mathrm{NE} } } = 2
 $
 other than \eqref{it:classification_of_twisted_sectors_of_codim_2_Elliptic} and \eqref{it:classification_of_twisted_sectors_of_codim_2_CM}.
 a general point of
 $
  W _{ \mathrm{NE} } \cap F ^{ - 1 } \lb \scM _{ 1, 3 } \rb
 $
 is represented by
 $
  (C, o, p, q)
 $,
 where
 $
  (C, o)
 $
 is a smooth elliptic curve and
 $
  p, q
 $
 are $2$-torsion points with respect to the origin $o$, equipped with the negation
 with respect to $o$.

 \item\label{it:classification_of_twisted_sectors_of_codim_2_1_2}
 For
 $
  ( g, n ) = ( 1, 2 )
 $,
other than \eqref{it:classification_of_twisted_sectors_of_codim_2_CM} there exists one more twisted sector
\(
   W _{ \mathrm{CM} }
\)
corresponding to the curves
$
 E _{ 1728 }
$
and
$
 E _{ 0 }
$
with the corresponding CM automorphism, equipped with two marked points which are fixed by the automorphism.
 \end{enumerate}
\end{summary}

\cref{table:The classification of twisted sectors of codimension 2} may be of use.
\begin{table}[h]
\begin{align}
\begin{array}{cccccc}
 \toprule
 g \backslash n & 0 & 1 & 2 & 3 & \ge 4\\
 \midrule
 1 & NA & NA
 & \eqref{it:classification_of_twisted_sectors_of_codim_2_1_2}
 & \eqref{it:classification_of_twisted_sectors_of_codim_2_1_3}
 & \eqref{it:classification_of_twisted_sectors_of_codim_2_generic}\\
 2
 & \eqref{it:classification_of_twisted_sectors_of_codim_2_2_0}
 & \eqref{it:classification_of_twisted_sectors_of_codim_2_2_1}
 & \eqref{it:classification_of_twisted_sectors_of_codim_2_2_2}
 & \eqref{it:classification_of_twisted_sectors_of_codim_2_generic}
 & \eqref{it:classification_of_twisted_sectors_of_codim_2_generic}\\
 3
 & \eqref{it:classification_of_twisted_sectors_of_codim_2_3_0}
 & \eqref{it:classification_of_twisted_sectors_of_codim_2_3_1}
 & \eqref{it:classification_of_twisted_sectors_of_codim_2_generic}
 & \eqref{it:classification_of_twisted_sectors_of_codim_2_generic}
 & \eqref{it:classification_of_twisted_sectors_of_codim_2_generic}\\
 4
 & \eqref{it:classification_of_twisted_sectors_of_codim_2_4_0}
 & \eqref{it:classification_of_twisted_sectors_of_codim_2_generic}
 & \eqref{it:classification_of_twisted_sectors_of_codim_2_generic}
 & \eqref{it:classification_of_twisted_sectors_of_codim_2_generic}
 & \eqref{it:classification_of_twisted_sectors_of_codim_2_generic} \\
 \ge 5
 & \eqref{it:classification_of_twisted_sectors_of_codim_2_generic}
 & \eqref{it:classification_of_twisted_sectors_of_codim_2_generic}
 & \eqref{it:classification_of_twisted_sectors_of_codim_2_generic}
 & \eqref{it:classification_of_twisted_sectors_of_codim_2_generic}
 & \eqref{it:classification_of_twisted_sectors_of_codim_2_generic} \\
 \bottomrule
\end{array}
\end{align}
\caption{The classification of twisted sectors of codimension 2}
\label{table:The classification of twisted sectors of codimension 2}    
\end{table}

\begin{proof}
This follows from
\cref{proposition:study_of_W_of_codim_2},
\cref{summary:for_the_classification_of_tw_sectors_of_codim_2},
and
\cref{summary:classify_quintuples}
(in particular, \cref{table:table_of_twisted_sectors_of_smooth_pointed_curves_codim2}).
\end{proof}

%
%
\section{Proof of the noncommutative rigidity}

In this last section we show
\cref{theorem:main theorem of the paper}, the main result of this paper.
Recall that the dimension of sheaf cohomology  is invariant under base field extensions.
Since the moduli stacks of pointed stable curves are already defined over
$
\bQ
$,
we can and will work over
$
\bC
$
throughout this section. See~\cref{section:introduction} for the structure of this section.

\subsection{The vanishing of \(H ^{ 2 } ( \cO )\)}
\label{section:vanishing_h2O}
In this section we prove
\begin{theorem}\label{theorem:vanishing_of_O}
The vanishing
\begin{equation}\label{eq:H1,2(O)}
H^i ( \MM , \cO_{\MM} ) = 0
\end{equation}
holds for
$
i = 1 , 2
$.
\end{theorem}
The case
$
g = 0
$
is trivial, since
$
\Mbar_{0,n}
$
are rational.
When
$
g = 1
$
and
$
n \le 2
$,
the assertion follows from the rationality of
$
\Mbar_{ 1 , 1 }
$
and
$
\Mbar_{ 1 , 2 }
$
(see \cite[Theorem 2.3]{MR3174737}).
In the rest of this subsection we assume
\begin{equation}\label{eq:technical_assumption}
g
\ge
2
\quad
\mathrm{or}
\quad
[
g=1
\
\mathrm{and}
\ 
n > 2],
\end{equation}
so that we always have
$
d\coloneqq
\dim
\scMbar_{g,n}
=
3g-3+n
\ge
3
$.
By
\cref{lemma:hodge_symmetry}, the claim of
\cref{theorem:vanishing_of_O}
is equivalent to the vanishing
$
H^0 ( \MM , \Omega^i_{\MM} ) = 0
$
for
$
i = 1 , 2
$. Below is the key ingredient of the proof.

\begin{lemma}
\label{lemma:topological_vanishing}
Under our assumption
\eqref{eq:technical_assumption},
$
 H^i (\cM_{g,n} , \bQ ) = 0
$
holds for
$
i=2d-1, 2d-2
$.
\end{lemma}

\begin{proof}
It was proven by Harer
(see
\cite[Chapter 19, Theorem (2.2)]{MR2807457})
that the vanishing
$
 H_i (M_{g,n} , \bQ ) = 0
$
holds when
\begin{equation}\label{eq:range_of_topological_vanishing}
i >
\begin{cases}
 n-3 & ( g = 0)
\\
4g-5 & ( g > 0 , n = 0 )
\\
4g-4 + n & ( g > 0 , n > 0 ).
\end{cases}
\end{equation}
We can easily check that
\eqref{eq:range_of_topological_vanishing}
holds for
$
i=2d-2
$
and hence also for
$
i=2d-1
$,
under our assumption
\eqref{eq:technical_assumption}.
Finally we have
$
H_i (M_{g,n} , \bQ )
\simeq
(H^i (M_{g,n} , \bQ ))^{\vee}
$
(see
\cite[Theorem A1]{MR0440554})
and
$
H^i (M_{g,n} , \bQ )
\simeq
H^i (\scM_{g,n} , \bQ )
$
(see
\cite[Proposition 36]{MR2172499}).
\end{proof}

\begin{proof}[Proof of \cref{theorem:vanishing_of_O}]
Let us deal with the case
$
i = 1
$ first.
Consider the exact sequence
\begin{equation}\label{eq:inductive_exact_sequence}
H^0 ( \MM , \Omega^1_{\MM}( \log \cB ) ( -\cB ) )
\to
H^0 ( \MM , \Omega^1_{\MM} )
\to
H^0 ( \cB^{1} , \Omega^1_{\cB^{1}} )
\end{equation}
obtained from
\eqref{eq:res_for_log_differential_forms}.
We check the vanishing of the first term of
\eqref{eq:inductive_exact_sequence}.
Using Serre duality and the perfect pairing
\eqref{eq:perfect_pairing}, we obtain
\begin{equation}
\begin{split}
H^0 ( \MM , \Omega^i_{\MM} ( \log \cB ) ( - \cB ) )
\simeq
H^0 ( \MM , \wedge^{d-i} \Theta_{\MM} ( - \log \cB ) \otimes \omega_{\MM} )
\\
\simeq
H^d ( \MM , \Omega^{d-i}_{\MM}( \log \cB ) )^{\vee}
\end{split}
\end{equation}
for any
$ i $. The dual of the last term is, by
\cref{lemma:log_hodge_decomposition},
a direct summand of
$
H^{2d-i} ( \cM_{g , n} , \bQ )
$.
As we saw in \cref{lemma:topological_vanishing}, this is always trivial
when
$
i = 1
$.

Next we show the vanishing of the third term of
\eqref{eq:inductive_exact_sequence}
by an induction on 
$
\dim \MM
$,
starting with the case when
$
\dim
\MM
=
3
$.
In the initial case we have
$
\dim
\cB'
=
2
$
and hence can check the assertion by hand,
since then
$
\cB'
$
is always rational.
For the general case, take any connected component
$
\cB'
\subset
\cB^{1}
$
and the finite \'etale cover
$
\pi
\colon
\cC'
\twoheadrightarrow
\cB'
$
of
\cref{lemma:inductive_structure_of_the_boundary}.
Since
$
\pi^*
\Omega^1_{\cB'}
\simeq
\Omega^1_{\cC'}
$,
it is enough to show
$
H^0 ( \cC' , \Omega^1_{\cC'} )
=
0
$.
Suppose
$
\cC' \simeq \scMbar_{g-1, n+2}.
$
Then we can use the induction hypothesis if
$
(g-1,n+2)
$
still satisfies
\eqref{eq:technical_assumption}, and otherwise
we already checked the assertion
in the beginning of this subsection.
If
$
\cC'
\simeq
\scMbar_{g' , n' } \times
\scMbar_{g'' , n'' }
$,
where
$
g' + g'' = g
$
and
$
n' + n'' = n + 2
$,
then by using the K\"unneth formula
we can check the assertion by the similar arguments.
Thus we conclude the proof for the case
$
i
=
1
$.

Finally we consider the case
$
i=2
$. The arguments are essentially the same
as in the case
$
i=1
$.
Consider the exact sequence
\begin{equation}\label{eq:inductive_exact_sequence_2}
	H^0 ( \MM , \Omega^2_{\MM}( \log \cB ) ( -\cB ) )
	\to
	H^0 ( \MM , \Omega^2_{\MM} )
	\to
	H^0 ( \cB^{1} , \Omega^2_{\cB^{1}} )
\end{equation}
obtained from \eqref{eq:res_for_log_differential_forms}.
We prove the vanishing of the middle term by an induction on
\(
	\dim \MM
\).

The triviality of the first term of \eqref{eq:inductive_exact_sequence_2} follows from the same arguments as above.
To show the vanishing of the third term using the induction hypothesis,
we consider the finite \'etale cover
$
\cC'
\twoheadrightarrow
\cB'
$
as before and check the vanishing of
$
 H^0 ( \cC' , \Omega^2_{\cC'} )
$.
As a new ingredient, when
$
\cC'
\simeq
\scMbar'
\times
\scMbar''
$,
we use the K\"unneth formula and need the result for
$
i=1
$
to show the vanishing of the term
\begin{equation}
H^0 ( \cC' , \Omega^1_{\scMbar'} \boxtimes \Omega^1_{\scMbar''} )
\simeq
H^0 ( \scMbar' ,  \Omega^1_{\scMbar'} )
\otimes
H^0 (  \scMbar'' , \Omega^1_{\scMbar''} ).
\end{equation}
Thus we conclude the proof.
\end{proof}

\subsection{No bivector fields, general case}
\label{section:proof_of_no_bivector_on_Mgnbar}

In this subsection we prove the following theorem.
\begin{theorem}\label{theorem:no_bivector_on_Mgnbar}
The vanishing
\begin{equation}\label{eq:H0_Theta_2_of_Mbar_g_n}
H^0 ( \MM , \wedge^2 \Theta_{\MM} ) = 0
\end{equation}
holds if
$
( g , n )
\neq
( 0 , 5 ),
( 1 , 2 )
$.
\end{theorem}

\begin{lemma}\label{lemma:vanishing_for_logarithmic_polyvector_fields}
Let
$
\scL
$
be a nef line bundle on
$
\scMbar_{g,n}
$.
Then
\begin{equation}\label{eq:Hq_Theta_p(-log)}
H^q ( \MM , \wedge^p \Theta_{\MM} (- \log \cB ) \otimes \scL^{-1})
=
0
\end{equation}
holds for any
$
p , q
$
with
$
p > q
$.
\end{lemma}

\begin{proof}
By Serre duality we obtain the following isomorphism.
\begin{equation}
\begin{split}
H^q (
\MM,
\wedge^p \Theta_{\MM} (- \log \cB ) \otimes \scL^{-1})^{\vee}
\simeq
H^{d-q}
\lb
\MM,
\Omega_{\MM}^p ( \log \cB ) (- \cB )
\otimes
\lb
\omega_{\MM } (\cB ) \otimes \scL
\rb
\rb
\end{split}
\end{equation}
Since 
$
\omega_{\MM } (\cB )
\otimes
\cL
$
is always ample by
\cref{proposition:ampleness_of_the_log_canonical_divisor},
by applying
\cref{theorem:logarithmic_KAN},
we see that the cohomology space is zero when
$
(d-q) + p >d
\iff
p > q
$.
\end{proof}

\begin{proof}[Proof of \cref{theorem:no_bivector_on_Mgnbar}]
We will first consider the general case of arbitrary $(p,q)$. 
Applying
\cref{lemma:spectral_sequence}
to
\eqref{eq:res_for_log_vector_fields},
we obtain the following spectral sequence.
\begin{equation}\label{eq:the_key_spectral_sequence}
\begin{split}
E^{s,t}_1
=
H^t
\lb
\cB^{s} , \wedge^{p-s} \Theta_{\cB^{s}}(\cE_{\cB^s})
\otimes
\det (\cN_{\cB^{s}/\MM})
\rb
\Rightarrow
E^{s+t}= H^{s+t}
\lb
\MM ,
\wedge^p \Theta_{\MM}(-\log \cB)
\rb
\end{split}
\end{equation}
Note that
$
E^{0, q}_1
=
H^q ( \scMbar_{g,n} , \wedge^p \Theta_{\scMbar_{g,n}} )
$.

By
\cref{lemma:vanishing_for_logarithmic_polyvector_fields},
if we have the inequality
$
p > q
$,
we immediately see
$
E^q = 0
$.
If moreover we have the vanishings
$
E_1^{1,q}
=
E_1^{2 , q-1}
=
E_1^{3 , q-2}
=
\cdots
=
E_1^{q+1 , 0}
=
0,
$
then we obtain the vanishing of
$E^{0,q}_1$.

In the rest we use this strategy to deal with the case
$
( p , q ) = ( 2 , 0 )
$
under the assumption
$
\dim
\MM
\ge 3
$.
In this case, since we want to show the vanishing of
$
E_1^{0,0}
$,
all we have to show is the vanishing of
\begin{equation}
E_1^{1,0}
=
H^0 ( \cB^{1} , \Theta_{\cB^{1}}(\cE_{\cB^1}) \otimes \cN_{\cB^{1} / \MM } )
=
H^0 ( \cB^{1} , \Theta_{\cB^{1}} \otimes \cN_{\cB^{1} / \MM } ).
\end{equation}
Note that
$
\cE_{\cB^1} \simeq \cO_{\cB^1}
$
as explained in the end of \cref{definition-proposition:orientation}. 

Let
$
\cB'
$
be a connected component of
$
\cB^{1}
$, and consider the finite \'etale Galois cover
$
\pi
\colon
\cC'
\twoheadrightarrow
\cB'
$
as in
\cref{lemma:inductive_structure_of_the_boundary}.
Since we have the inclusion
\begin{equation}
\pi^*
\colon
H^0 ( \cB' , \Theta_{\cB'} \otimes \cN_{\cB' / \MM } )
\hto
H^0 ( \cC' , \Theta_{\cC'} \otimes \pi^*\cN_{\cB' / \MM } ),
\end{equation}
it is enough to show the vanishing of the RHS.

If
$
\cC'
\simeq
\scMbar'
\times
\scMbar''
$,
then by
\cite[Lemma 4.2]{MR2460695}
and the K\"unneth formula,
the RHS is isomorphic to
\begin{equation}\label{eq:kunneth}
\begin{split}
\lb
H^0
 ( \scMbar' , \Theta_{\scMbar'} \otimes \psi_{n+1}^{\vee} )
\otimes
H^0 ( \scMbar'' , \psi_{n+2}^{\vee} )
\rb
\oplus \ 
\lb
H^0 ( \scMbar'' , \Theta_{\scMbar''} \otimes \psi_{n+2}^{\vee} )
\otimes
H^0 ( \scMbar' , \psi_{n+1}^{\vee} )
\rb.
\end{split}
\end{equation}
Since the line bundles
$
\psi_i
$
are nef and big by \cref{lemma:psi_classes_are_nef_and_big},
the space \eqref{eq:kunneth} is trivial when
$
\dim
\scMbar'
$
and
$
\dim
\scMbar''
$
are both positive (cf. \cite[Theorem A.1]{MR2460695}).
Otherwise we can use
\cref{proposition:vanishing_of_H0theta1_anti_nef}
below.

If
$
\cC'
\simeq
\scMbar_{g-1 , n+2}
$,
again by
\cite[Lemma 4.2]{MR2460695},
the RHS is isomorphic to
\begin{equation}
H^0 ( \scMbar_{g-1 , n+2} , \Theta_{\scMbar_{g-1 , n+2}}
\otimes \psi_{n+1}^{\vee}
\otimes
\psi_{n+2}^{\vee} ).
\end{equation}
Since
$
\psi_{n+1}
\otimes
\psi_{n+2}
$
is nef and big by
\cref{lemma:psi_classes_are_nef_and_big},
we can use
\cref{proposition:vanishing_of_H0theta1_anti_nef}
again to conclude the proof.
\end{proof}

\begin{proposition}\label{proposition:vanishing_of_H0theta1_anti_nef}
For any anti-nef line bundle
$
\cL
$,
$
 H^0 ( \MM , \Theta_{\MM} \otimes \cL )
=
0
$
holds if
$
\dim
\MM
\ge 2
$.
\end{proposition}

\begin{proof}
By applying
$
\otimes \cL
$
to
\eqref{eq:res_for_log_vector_fields}
with
$
p = 1
$, we obtain the exact sequence
\begin{equation}
\begin{split}
H^0 ( \MM , \Theta_{\MM} (- \log \cB ) \otimes \cL )
\to
H^0 ( \MM , \Theta_{\MM} \otimes \cL )
\to
H^0 ( \cB^{1} , \cL \otimes \cE_{\cB^1} \otimes \cN_{\cB^{1} / \MM } ).
\end{split}
\end{equation}
Since
$
\cL \otimes \cE_{\cB^1} \otimes \cN_{\cB^1 / \MM}
$
is anti-nef and big
on the finite \'etale cover
$
\cC'
$
of
\cref{lemma:inductive_structure_of_the_boundary}
and
$
\dim
\cB^{1}
\ge 1
$,
the third term has to be trivial.
Finally, the vanishing of the first term
follows from \cref{lemma:vanishing_for_logarithmic_polyvector_fields}.
\end{proof}

\subsubsection{Interlude}
Since \cref{theorem:logarithmic_KAN} is valid over arbitrary perfect fields
and the moduli stacks
$
\scMbar_{g,n}
$
are defined over prime fields, one can easily check that
\cref{theorem:no_bivector_on_Mgnbar} is generalized to
positive characteristics as follows.

\begin{theorem}
Let
$
\bfk
$
be a field of characteristic
$
p > 0
$, and
suppose
$
(g,n)
$
satisfies the inequalities
$
3 \le 3g-3+n \le p.
$
Then
$
H^0(\scMbar_{g,n}, \wedge^2\Theta_{\scMbar_{g,n}})=0
$.
\end{theorem}
\begin{proof}
The inequality
$
3g-3+n \le p
$
ensures the equality
$
\bfd(\dim \cB^i, p)
=
\dim \cB^i
$
for all strata
$
\cB^i
$
of
$
\scMbar_{g,n}
$.
Therefore the arguments in the case of characteristic \(0\)
works without change.
\end{proof}

\subsection{Exceptional untwisted sectors
$
\Mbar_{0,5}
$
and
$
\scMbar_{1,2}
$}\label{section:exceptional_cases}

Recall first that
$
\Mbar_{0,5}
$
is isomorphic to the blowup of
$
\bP^2
$
in four points in general position. This immediately tells us

\begin{lemma}\label{lemma:bivectors_on_M05}
$
\dim H^0 ( \Mbar_{0,5} , \wedge^2 \Theta_{\Mbar_{0,5}})
=
6.
$
\end{lemma}

We next prove
\begin{lemma}\label{lemma:HH3(M05bar)_is_trivial}
$
\HH^3(\Mbar_{0,5})=0
$.
\end{lemma}

\begin{proof}
By the HKR isomorphism, we see that
$
\HH^3(\Mbar_{0,5})
$
is isomorphic to the direct sum
\begin{equation}\label{eq:HKR_for_H3_M05bar}
\begin{split}
H^3 ( \Mbar_{0,5} , \cO_{\Mbar_{0,5}} )
\oplus
H^2 ( \Mbar_{0,5} , \Theta_{\Mbar_{0,5}} )
\oplus
H^1 ( \Mbar_{0,5} , \wedge^2 \Theta_{\Mbar_{0,5}} )
\oplus
H^0 ( \Mbar_{0,5} , \wedge^3 \Theta_{\Mbar_{0,5}} ).
\end{split}
\end{equation}
Since
$
\dim \Mbar_{0,5} = 2,
$
we see
$
H^3 ( \Mbar_{0,5} , \cO_{\Mbar_{0,5}} )
=
0
$
and
$
\wedge^3 \Theta_{\Mbar_{0,5}}
=
0
$.
Moreover, using the canonical isomorphism
$
\Theta_{\Mbar_{0,5}}
\simeq
\Omega^1_{\Mbar_{0,5}} \otimes \omega_{\Mbar_{0,5}}^{-1}
$
and applying
\cref{theorem:logarithmic_KAN}
to the ample line bundle
$
\cL
=
\omega_{\Mbar_{0,5}}^{-1}
$,
we can show the vanishing of the two middle terms of
\eqref{eq:HKR_for_H3_M05bar}.
\end{proof}

Recall next from
\cite[Section 2]{MR3174737}
the description of
$
\scMbar_{1,2}
$
as a global quotient stack.
Consider the smooth variety
\begin{equation}\label{eq:definition_of_atlas_of_X}
X
\coloneqq
Z
( z y^2 - x^3 - a x z^2 - b z^3 )
\subset
\bA^3_{0} \times \bA^2_{0}.
\end{equation}
Here
$
\bA^*_{0}
\coloneqq
\bA^* \setminus \{ 0 \}
$,
$
( x , y , z )
$
are the coordinates of
$
\bA^3
$,
and
$
( a , b )
$
are the coordinates of
$
\bA^2
$.
The torus
$
G
=
\bG_m^2
$
acts on
$
\bA^3_{0} \times \bA^2_{0}
$
with the weights in the table below.

\[
\begin{array}{c|ccccc}
& x & y & z & a & b
\\
\hline
\xi & 1 & 1 & 1 & 0 & 0
\\
\lambda & 2 & 3 & 0 & 4 & 6
\end{array}
\]
Here
$
( \xi , \lambda )
$
are the coordinates of
$
G
$. Note that the defining equation
\begin{equation}
f
\coloneqq
z y^2 - x^3 - a x z^2 - b z^3
\end{equation}
is homogeneous of weight
$
( 3 , 6 )
$.
Then we have a natural isomorphism of stacks
\begin{equation}\label{eq:M1,2_as_global_quotient}
[ X / G ]
\simto
\scMbar_{1,2}.
\end{equation}

Under the canonical equivalence
\begin{equation}\label{eq:coherent_sheaves_as_equivariant_sheaves}
\pi^* \colon \coh [X/G]
\simto
\coh^G X,
\end{equation}
we see the coincidence (as linearized line bundles)
$
\pi^* \omega_{[X/G]}
\simeq
\omega_X
$.

\begin{proposition}\label{proposition:bivectors_on_M12}
$
H^0 ( \scMbar_{1,2} , \wedge^2 \Theta_{\scMbar_{1,2}} )
=
0.
$
\end{proposition}

\begin{proof}
Set
$
U \coloneqq
\bA^3_{0} \times \bA^2_{0}
$
for simplicity.
By the adjunction formula, we see
\begin{equation}
\omega_{[X/G]}
\simeq
\omega_{[U/G]}|_{[X/G]}
\otimes
\cO_{[X/G]} ( [X/G] ).
\end{equation}
Under the equivalence
\eqref{eq:coherent_sheaves_as_equivariant_sheaves},
it is translated into
\begin{equation}
\begin{split}
\pi^* \omega_{[X/G]}
\simeq
\omega_X
\simeq
\omega_U|_X
\otimes
\cO_X (X)
\simeq
\lb
\cO_X \cdot dx \wedge dy \wedge dz \wedge da \wedge db
\rb
\otimes
\lb
\cO_X \cdot f
\rb^{\vee}
\\
\simeq
\cO_X (-3, -15) \otimes \cO_X (3 , 6)
\simeq
\cO_X (0 , -9).
\end{split}
\end{equation}
Therefore we see
\begin{equation}
H^0 ( [X/G] , \wedge^2 \Theta_{[X/G]})
\simeq
\Hom_X^G (\cO_X , \cO_X (0 , 9))
\simeq
\lb
\bfk [x,y,z,a,b]/(f)
\rb_{(0,9)}
=
0.
\end{equation}
The 2nd isomorphism follows from the normality of the hypersurface defined by
\(
   f
\)
and that the codimension of the complement of
\(
   X
\)
in the hypersurface is
\(
   2
\).
The last equality is due to the
absence of monomials of bidegree
$
(0,9)
$.
\end{proof}

\subsection{The vanishing of contributions from the twisted sectors}
\label{section:The_vanishing_of_contributions_from_the_twisted_sectors}
We first discuss the vanishing of the contributions from the twisted sectors
$
 W \subset I \scMbar _{ g, n }
$
of
$
 c _{ W } = 1
$.

\begin{proposition}\label{proposition:vanishing_of_contributions_from_twisted_sectors_of_codim_1}
Let
$
 W \subset I \scMbar _{ g, n }
$
be the twisted sector with
$
 c _{ W } = 1
$
of common type as in \cref{summary:classification_of_twisted_sectors_of_codim_2}. 
Assume that $(g,n) \neq (1,2) $. 
Then the following cohomology vanishings hold.
\begin{enumerate}[(i)]
\item\label{it:H1}
 $H^1(W, \cN_{W/ \scMbar_{g,n}})=0$. 

\item\label{it:H0}
 $H^0(W, \Theta_{W} \otimes \cN_{W/ \scMbar_{g,n}})=0$. 
\end{enumerate}
\end{proposition}

\begin{proof}
Let
$
 \rho \colon \scMbar_{1,1} \times \scMbar_{g-1, n+1} \to W
$
be the finite \'{e}tale cover as in \cref{lemma:inductive_structure_of_the_boundary}.
Set
$
 X_1\coloneqq \scMbar_{1,1}, X_2\coloneqq \scMbar_{g-1,n+1}
$, and let
$
 \pr_i \colon X_1 \times X_2 \to X_i
$
be the $i$-th projection for $i = 1, 2$.

\noindent\eqref{it:H1} It is enough to show
$
 H^1(X_1 \times X_2, \rho^* \cN_{W/ \scMbar_{g,n}} ) =0
$.
Since we have 
$
 \rho^* \cN_{W/ \scMbar_{g,n}}
 \simeq
 \pr_1^* \psi_1^{\vee} \otimes \pr_2^* \psi_{n+1}^{\vee}
$ by \cref{lemma:inductive_structure_of_the_boundary},
we obtain
\begin{align}
\begin{aligned}
H^1(X_1 \times X_2, \rho^* \cN_{W/ \scMbar_{g,n}} )
\simeq H^0(X_1, \psi_1^{\vee}) \otimes H^1(X_2, \psi_{n+1}^{\vee})\\
\oplus  H^1(X_1, \psi_1^{\vee}) \otimes H^0(X_2, \psi_{n+1}^{\vee}).
\end{aligned}
\end{align}
This is zero since $\psi_1, \psi_{n+1}$ are nef and big by \cref{lemma:psi_classes_are_nef_and_big} 
and $\dim \scMbar_{g-1, n+1} >0$ by $(g,n) \neq (1,2)$ (cf. \cite[Theorem A.1]{MR2460695}). 

\noindent\eqref{it:H0} It is enough to show that
$
 H^0 ( X_1 \times X_2, \Theta_{X_1 \times X_2} \otimes
 \rho ^* \cN_{W/ \scMbar_{g,n}})=0
$.
This cohomology group is isomorphic to
$
 H^0 ( X_1 \times X_2, (\pr_1^* \Theta_{X_1} \oplus \pr_2^* \Theta_{X_2} )
 \otimes (\pr_1^* \psi_1^{\vee} \otimes \pr_2^* \psi_{n+1}^{\vee}))
$,
which is zero since we have
$
 H^0(X_1, \psi_1^{\vee}) = 0
$
and
$
 H^0( X_2, \psi_{n+1}^{\vee} ) = 0
$
by \cref{lemma:psi_classes_are_nef_and_big} and $\dim \scMbar_{g-1, n+1} >0$ by $(g,n) \neq (1,2)$.
\end{proof}

Next we compute the cohomology group
$
 H^0 ( W, \wedge^2 \cN _{ W / \scMbar _{ g, n } } )
$
for the twisted sectors
$
 W \subset I \scMbar _{ g, n }
$
of
$
 c _{ W } = 2
$
such that
$
 F ( W ) \subset \cB \subset \scMbar _{ g, n }
$
as follows. 

\begin{proposition}\label{proposition:vanishing for codim 2 sectors}
Let
$
 W \subset I \scMbar _{ g, n }
$
be a twisted sector of
$
 c _{ W } = 2
$
of type \eqref{it:classification_of_twisted_sectors_of_codim_2_HE},
\eqref{it:classification_of_twisted_sectors_of_codim_2_Elliptic},
\eqref{it:classification_of_twisted_sectors_of_codim_2_CM}, or
\eqref{it:classification_of_twisted_sectors_of_codim_2_3_components} in 
\cref{summary:classification_of_twisted_sectors_of_codim_2}.
Then
\begin{enumerate}[(I)]
    \item 
    \(
        H ^{ 0 }
        \left(
            W _{ \mathrm{ I } },
            \wedge^2 \cN _{ W / \scMbar _{ g, n } }
        \right)
        =
        0
    \)
    if 
    \(
        ( g, n ) \neq ( 2, 2 )       
    \).
    \item 
    \(
        H ^{ 0 }
        \left(
            W _{ \mathrm{ II } },
            \wedge^2 \cN _{ W / \scMbar _{ g, n } }
        \right)
        =
        0
    \)
    if 
    \(
        ( g, n ) \neq ( 1, 3 )       
    \).
    \item 
    \(
        H ^{ 0 }
        \left(
            W _{ \mathrm{ III } },
            \wedge^2 \cN _{ W / \scMbar _{ g, n } }
        \right)
        =
        0
    \)
    if 
    \(
        ( g, n ) \neq ( 1, 2 )
    \).
    \item 
    \(
        H ^{ 0 }
        \left(
            W _{ \mathrm{ IV } },
            \wedge^2 \cN _{ W / \scMbar _{ g, n } }
        \right)
        =
        0
    \)
    for all 
    \( ( g, n ) \).
\end{enumerate}
\end{proposition}

\begin{proof}
\eqref{it:classification_of_twisted_sectors_of_codim_2_HE} Let $X_1\coloneqq \scMbar_{2,1}$, $X_2\coloneqq \scMbar_{g-2, n+1}$ and 
$W ^{ 1 } \subset X_1$ be the Weierstrass divisor.
Let
$
 w \colon \Wtilde \coloneqq W \times _{ \cB _{ 2 : \emptyset  } } \lb X _{ 1 } \times X _{ 2 } \rb \to W
$
be the base change morphism, which is easily seen to be finite and \'etale. Since 
$
 W ^{ 1 } \times X _{ 2 }
$
is smooth, the canonically induced morphism
$
 \Wtilde \to W ^{ 1 } \times X _{ 2 }
$
is an isomorphism. Hence we have an isomorphism 
\begin{align}
\begin{aligned}
 w ^{ * } \wedge^2 \cN _{ W / \scMbar _{ g, n } }
 \simeq
 \wedge^2 \cN _{ \Wtilde / \scMbar _{ g, n } }
 \simeq \cN_{ W ^{ 1 } \times X _{ 2 } / X _{ 1 } \times X _{ 2 } }
 \otimes \cN_{ X_1 \times X_2 / \scMbar _{ g, n } }\\ 
 \simeq \pr_1 ^{ * } \cN_{ W ^{ 1 } / X_1} \otimes \lb \psi_1^{\vee} \boxtimes \psi_{n+1}^{\vee} \rb\\
  \simeq ( \cN_{W ^{ 1 } / X_1} \otimes \psi_1^{\vee} | _{ W ^{ 1 } } ) \boxtimes \psi_{n+1}^{\vee},   
\end{aligned}
\end{align}
where
$
 \pr_1 \colon W ^{ 1 } \times X_2 \to W ^{ 1 }
$
is the first projection. Thus, by the K\"unneth formula and the fact that $\psi_{n+1}$ is nef and big, we obtain
$
 H^0 ( W, \wedge^2 \cN _{ W / X } ) =0
$ except when $(g,n) =(2,2)$.

\noindent\eqref{it:classification_of_twisted_sectors_of_codim_2_Elliptic} Let $W ^{ 2 } \subset \scMbar_{1,2}$ be the divisor whose general point represents a pointed curve
$
 (C; p_1, p_2)
$
such that 
$
 \cO_C ( p_1 - p_2 )
$
is $2$-torsion. 
Then the twisted sector $W$ in this case admits a birational finite morphism onto the divisor
$
 W ^{ 2 } \times \scMbar_{g-1,n}
 \subset 
 \scMbar_{1,2} \times \scMbar_{g-1,n}
$. By this description, we can compute $ H^0 ( W, \wedge^2 \cN _{ W / X } ) =0$ as in \eqref{it:classification_of_twisted_sectors_of_codim_2_HE} except when $(g,n) = (1,3)$. 

\noindent\eqref{it:classification_of_twisted_sectors_of_codim_2_CM} Let $W ^{ 3 } \subset \scMbar_{1,1}$ be a divisor (point) corresponding to either
$
 E _{ 1728 }
$
or
$
 E _{ 0 }
$. 
Then, in this case, $W$ is a divisor $W ^{ 3 } \times \scMbar_{g-1,n+1} \subset 
\scMbar_{1,1} \times \scMbar_{g-1,n+1}$. Hence we can also show
$
 H^0 ( W, \wedge^2 \cN _{ W / X } ) = 0
$
as in \eqref{it:classification_of_twisted_sectors_of_codim_2_HE} except when
$
 ( g, n ) = ( 1, 2 )
$; one can also settle this case by noting that
$
 \cN _{ W ^{ 3 } / \scM _{ 1, 1 }  }
$
corresponds to the nontrivial characters of the stabilizer groups of the points corresponding to
$
 E _{ 1728 }
$
and
$
 E _{ 0 }
$.

\noindent\eqref{it:classification_of_twisted_sectors_of_codim_2_3_components}
Let
$
 W ^{ 4 }
$
be the normalization of the divisor
$
 \cB _{ 1: 1 } \subset \scMbar_{g-1, n+1}
$, which hence admits an isomorphism
$
 \scMbar_{g-2, n+2} \times \scMbar_{1,1} \simeq W _{ 4 }
$.
With these notation, $W$ in this case is isomorphic to
$
 \scMbar_{1,1} \times W ^{ 4 } \lb \subset \scMbar_{1,1} \times \scMbar_{g-1, n+1} \rb
$.
Hence we can also check
$
 H^0 ( W, \wedge^2 \cN _{ W / X } ) =0
$
as in \eqref{it:classification_of_twisted_sectors_of_codim_2_HE} by using that $\psi_1$ is nef and big on $\scMbar_{1,1}$.  
\end{proof}

Combining the results obtained above, we get the following conclusion.
\begin{summary}
The second Hochschild cohomology of
$
 \Mbar _{ 0, n }
$
vanishes for all
$
 n \neq 5
$,
and it is of dimension 6 when $n = 5$. For
$
 g \ge 1
$,
the second Hochschild cohomology of
$
 \scMbar _{ g, n }
$
vanishes for all
$
 ( g, n )
$
except possibly for the following 8 cases:
\begin{align}\label{eq:the_list_of_possible_exceptions}
 ( g, n ) = ( 4, 0 ),
 ( 3, 1 ), ( 3, 0 ),
 ( 2, 2 ), ( 2, 1 ), ( 2 , 0 ),
 ( 1, 3 ), ( 1, 2 ).
\end{align}
\end{summary}

\begin{proof}
    We have already shown everything except the vanishing for the case
$
 ( g, n ) = ( 1, 1 )
$.
Since
$
 \dim \scMbar _{ 1, 1 } = 1
$,
there is no contribution from the untwisted sector by Hacking's rigidity.
There is only one twisted sector of codimension 0 which is isomorphic to the untwisted sector as we saw in~\cref{proposition:classification_of_twisted_sectors_of_codim_0}, which hence does not contribute either.
The twisted sectors of codimension 1 are classified in \cref{summary:classification_of_twisted_sectors_of_codim_1}. They are essentially the same as \eqref{it:classification_of_twisted_sectors_of_codim_2_CM} of \cref{summary:classification_of_twisted_sectors_of_codim_2}, and the proof of~\cref{proposition:vanishing for codim 2 sectors} for the case \eqref{it:classification_of_twisted_sectors_of_codim_2_CM}, especially the argument about the divisor
\(
   W ^{ 3 } \subset \scMbar_{1,1}
\),
literally works.
\end{proof}
\subsection{The remaining cases}
\label{section:The_remaining_cases}
We expect that the 8 exceptions in \eqref{eq:the_list_of_possible_exceptions} are merely technical, in the sense that the vanishing of the second Hochschild cohomology should be the case also in those cases. In order to show the vanishing, we need more negativity results for the normal bundles and the tangent bundles of particular loci in
$
 \scMbar _{ g, n }
$
for lower
$
 ( g,n )
$.

Below we summarize what we have understood in so far, by describing
$
 \HH ^{ 2 } \lb \scMbar _{ g, n } \rb
$
as a direct sum of sheaf cohomology groups (which we expect to vanish) for
$
 ( g, n )
$
as in \eqref{eq:the_list_of_possible_exceptions}.

\begin{corollary}\label{corollary:descriptions_of_the_remaining_cases}
We have the following isomorphisms for the second Hochschild cohomology groups of
$
 \scMbar _{ g, n } 
$
for the 8 exceptional
$
 ( g, n )
$
of \eqref{eq:the_list_of_possible_exceptions}. We use notation for the twisted sectors as in \cref{summary:classification_of_twisted_sectors_of_codim_2} and \cref{summary:classification_of_twisted_sectors_of_codim_2}.

\begin{enumerate}
\item\label{it:description_of_hh2_4_0}
$
 \HH ^{ 2 } \lb \scMbar _{ 4, 0 } \rb
 \simeq
 H ^{ 0 } \lb W _{ \mathrm{HE} }, \wedge ^{ 2 } \cN _{ W _{ \mathrm{HE} } / \scMbar _{ 4, 0 } } \rb
$.

\item\label{it:description_of_hh2_3_1}
$
 \HH ^{ 2 } \lb \scMbar _{ 3, 1 } \rb
 \simeq
 H ^{ 0 } \lb W _{ \mathrm{HE} }, \wedge ^{ 2 } \cN _{ W _{ \mathrm{HE} } / \scMbar _{ 3, 1 } } \rb
$.

\item\label{it:description_of_hh2_3_0}
$
 \HH ^{ 2} \lb \scMbar _{ 3, 0 } \rb
 \simeq
 H ^{ 1 } \lb W _{ \mathrm{HE} }, \cN _{ 
    W _{ \mathrm{HE} } / \scMbar _{ 3, 0 } } \rb
    \\
 \oplus
 H ^{ 0 } \lb W _{ \mathrm{HE} },
 \Theta _{ W _{ \mathrm{HE} } } \otimes \cN _{ 
    W _{ \mathrm{HE} } / \scMbar _{ 3, 0 } } \rb
 \oplus
 H ^{ 0 } \lb W _{ \mathrm{BE} }, \wedge ^{ 2 } \cN _{ W _{ \mathrm{BE} } / \scMbar _{ 3, 0 } } \rb
$.

\item\label{it:description_of_hh2_2_2}
$
 \HH ^{ 2 } \lb \scMbar _{ 2, 2 } \rb
 \simeq
 H ^{ 0 } \lb W _{ \mathrm{I} }, \wedge ^{ 2 } \cN _{ W _{ \mathrm{I} } / \scMbar _{ 2, 2 } } \rb
 \oplus
 H ^{ 0 } \lb W _{ \mathrm{HE} }, \wedge ^{ 2 } \cN _{ W _{ \mathrm{HE} } / \scMbar _{ 2, 2 } } \rb
$.

\item\label{it:description_of_hh2_2_1}
\(
 \HH ^{ 2 } \lb \scMbar _{ 2, 1 } \rb
 \simeq
 H ^{ 1 } \lb W _{ \mathrm{HE} }, \cN _{ W _{ \mathrm{HE} } / \scMbar _{ 2, 1 } } \rb\\
 \oplus
 H ^{ 0 } \lb W _{ \mathrm{HE} }, \Theta _{ W _{ \mathrm{HE} } } \otimes \cN _{ W _{ \mathrm{HE} } / \scMbar _{ 2, 1 } } \rb
 \oplus
 H ^{ 0 } \lb W _{ \mathrm{BE} }, \wedge ^{ 2 } \cN _{ W _{ \mathrm{BE} } / \scMbar _{ 2, 1 } } \rb.   
\)

\item\label{it:description_of_hh2_2_0}
\(
 \HH ^{ 2 } \lb \scMbar _{ 2, 0 } \rb
 \simeq
  H ^{ 1 } \lb W _{ \mathrm{BE} }, \cN _{ W _{ \mathrm{BE} } / \scMbar _{ 2, 0 } } \rb\\
 \oplus
 H ^{ 0 } \lb W _{ \mathrm{BE} }, \Theta _{ W _{ \mathrm{BE} } } \otimes \cN _{ W _{ \mathrm{BE} } / \scMbar _{ 2, 0 } } \rb
 \oplus
 H ^{ 0 } \lb W _{ 3 }, \wedge ^{ 2 } \cN _{ W _{ 3 } / \scMbar _{ 2, 0 } } \rb
 \oplus\\
 H ^{ 0 } \lb W _{ 4 }, \wedge ^{ 2 } \cN _{ W _{ 4 } / \scMbar _{ 2, 0 } } \rb
 \oplus
 H ^{ 0 } \lb W _{ 6 }, \wedge ^{ 2 } \cN _{ W _{ 6 } / \scMbar _{ 2, 0 } } \rb
 \oplus
 H ^{ 0 } \lb W _{ \mathrm{CM} }, \wedge ^{ 2 } \cN _{ W _{ \mathrm{CM} } / \scMbar _{ 2, 0 } } \rb.   
\)

\item\label{it:description_of_hh2_1_3}
\(
 \HH ^{ 2 } \lb \scMbar _{ 1, 3 } \rb
 \simeq
 H ^{ 0 } \lb W _{ \mathrm{II} }, \wedge ^{ 2 } \cN _{ W _{ \mathrm{II} } / \scMbar _{ 1, 3 } } \rb
 \oplus
 H ^{ 0 } \lb W _{ \mathrm{NE} }, \wedge ^{ 2 } \cN _{ W _{ \mathrm{NE} } / \scMbar _{ 1, 3 } } \rb.   
\)

\item\label{it:description_of_hh2_1_2}
Let
$
 W \subset I \scMbar _{ 1, 2 }
$
be the twisted sector of
$
 c _{ W} = 1
$
as in \cref{proposition:study_of_W_of_codim_1}. Then
\(
\HH ^{ 2 } \lb \scMbar _{ 1, 2 } \rb
 \simeq
 H ^{ 1 } \lb W, \cN _{ W / \scMbar _{ 1, 2 } } \rb
 \oplus
 H ^{ 0 } \lb W, \Theta _{ W } \otimes \cN _{ W / \scMbar _{ 1, 2 } } \rb
 \oplus 
  H ^{ 1 } \lb W _{ \mathrm{NE} }, \cN _{ W _{ \mathrm{NE} } / \scMbar _{ 1, 2 } } \rb
 \oplus
 H ^{ 0 } \lb W _{ \mathrm{NE} }, \Theta _{ W _{ \mathrm{NE} } } \otimes \cN _{ W _{ \mathrm{NE} } / \scMbar _{ 1, 2 } } \rb
 \oplus
 H ^{ 0 } \lb W _{ \mathrm{CM} }, \wedge ^{ 2 } \cN _{ W _{ \mathrm{CM} } / \scMbar _{ 1, 2 } } \rb
 \oplus
 H ^{ 0 } \lb W _{ \mathrm{III} }, \wedge ^{ 2 } \cN _{ W _{ \mathrm{III} } / \scMbar _{ 1, 2 } } \rb   
\).
\end{enumerate}
\end{corollary}

\bibliographystyle{amsalpha}
\bibliography{mainbibs}

\end{document}